\documentclass[11pt]{amsart}
\usepackage{amsmath,amssymb,amsthm,upref,graphicx,mathrsfs, enumerate}
\usepackage{esint}
\usepackage{textcomp}
\usepackage{array}
\usepackage{color,xcolor}
\usepackage[
 colorlinks=true,
 linkcolor=blue,
 citecolor=blue,
 urlcolor=blue,
 pagebackref]{hyperref}

\numberwithin{equation}{section}
\allowdisplaybreaks

\newtheorem{theorem}{Theorem}[section]

\newtheorem{lemma}[theorem]{Lemma}
\newtheorem{cor}[theorem]{Corollary}

\theoremstyle{definition}

\newtheorem{remark}[theorem]{Remark}

\newcommand{\Mm}{\mathcal{M}}

\newcommand{\Rn}{\mathbb{R}^n}
\newcommand{\Z}{\mathbb{Z}}

\newcommand{\supp}{{\rm supp}\,}

\newcommand{\f}{\frac}
\newcommand{\vc}{\infty}

\def\Z{\mathbb{Z}}

\def\Rn{\mathbb{R}^n}
\def\S{\mathcal{S}}

\def\Rnm{\mathbb{R}^{nm}}
\def\Rt{\mathbb{R}^{2n}}
\def\supp{\operatorname{supp}}
\def\dist{\operatorname{dist}}

\def\XXint#1#2#3{{\setbox0=\hbox{$#1{#2#3}{\int}$}
\vcenter{\hbox{$#2#3$}}\kern-.5\wd0}}

\begin{document}

\title[On the weak boundedness of multilinear Littlewood--Paley functions]
 {On the weak boundedness of multilinear Littlewood--Paley functions}

\authors

\author[Hormozi]{Mahdi Hormozi}
\address{Mahdi Hormozi\\
School of Mathematics\\
Institute for Research in Fundamental Sciences (IPM)\\
P. O. Box 19395-5746, Tehran, Iran} \email{me.hormozi@gmail.com}

\author[Sawano]{Yoshihiro Sawano}
\address{Yoshihiro Sawano\\
Department of Mathematics\\
Chuo University\\
1-13-27, Kasuga, 112-8551, Tokyo, Japan} \email{yoshihiro-sawano@celery.ocn.ne.jp}

\author[Yabuta]{K\^{o}z\^{o} Yabuta}
\address{K\^{o}z\^{o} Yabuta\\
Research Center for Mathematics and Data Science\\
Kwansei Gakuin University\\
Gakuen 2-1, Sanda 669-1337\\
Japan}\email{kyabuta3@kwansei.ac.jp}

\thanks{M. Hormozi is supported by a grant from IPM. Y. Sawano was supported by Grant-in-Aid for Scientific Research (C) (19K03546), 
the Japan Society for the Promotion of Science and People's Friendship University of Russia.}


\arraycolsep=1pt

\begin{abstract}
 In this note, notwithstanding the generalization, 
 we simplify and shorten the proofs of the main results 
 of 
 the third author's paper
 \cite{SXY} significantly. In particular, the new proof for 
 \cite[Theorem 1.1]{SXY} is quite short and, unlike the original proof, 
 does not rely on the properties of 
 the ``Marcinkiewicz function''. 
 This allows us to get a precise linear dependence on Dini constants with 
 a subsequent application to Littlewood--Paley operators by well-known 
 techniques. In other words, we relax the
log-Dini condition in the pointwise bound to the classical Dini condition.
This solves an open problem (see e.g. \cite[pp. 37--38]{CY}). 
Our method can be applied to the multilinear case.
\end{abstract}
\maketitle
\section{Introduction}

We seek a sharp sparse estimate for Littlewood--Paley operators.
The class of integral kernels is wide in the sense that
we allow the moduli $\varphi$ and $w$ of continuity
to satisfy the Dini condition
$\int\limits_{0}^{1} \frac{\varphi(t)+w(t)}{t}dt<\infty$.
Here by a modulus of continuity,
we mean a positive increasing function on $(0,1)$.
Write $\Gamma_\alpha(x)$ for the cone in $\mathbb{R}^{n+1}_{+}$ of aperture $\alpha>1$
centered at $x$, that is,
where 
\begin{equation}\label{eq:220403-1111}
\Gamma_\alpha(x)=\{(y,t)\in {\mathbb R}^{n+1}_+: |x-y|<\alpha t\}.
\end{equation}
Let $S_{\alpha,\phi}$ be the square function defined 
by means of a standard kernel $\phi$ as follows:
\begin{equation}\label{eq:220330-2}
S_{\alpha, \phi} f (x)=\Big(\iint_{\Gamma_\alpha(x)}|f\star \phi_t(y)|^2\f{dydt}{t^{n+1}}\Big)^{\frac12},
\end{equation}
 where $\phi_t(x)=t^{-n} \phi(x/t)$ and $\star$ refers 
to the convolution operation of two functions. 

The study on the linear/multilinear square functions has important applications 
in PDEs and other fields
of mathematics. 
For further details on the theory of linear multilinear square functions and their applications, 
we refer to \cite{CHIRSY, Ler5, Ler6, SXY} and the references therein.

In \cite{Ler6}, Lerner 
proved sharp weighted norm inequalities for $S_{\alpha, \phi} f $ 
by applying \textit{intrinsic square functions} introduced in \cite{Wil}. 
Later on, Lerner himself improved the result---he 
obtained the sharp dependence on $\alpha$---in \cite{Ler5} by using the local mean 
oscillation formula. 
Motivated by these works, 
many authors obtained many important results
 (see e.g. \cite{BBD,BuiNew,BuiHormozi,CY}). 
Recall that 
a modulus of continuity is an increasing concave function defined on $(0,1)$.
Let $\varphi:[0,1] \to [0,\infty)$ be a modulus of continuity which satisfies 
the Dini condition.
That is,
$\varphi:[0,1] \to [0,\infty)$
is an increasing function such that
$[\varphi]_{\rm{Dini}}:=\int_{0}^{1} \frac{\varphi(t)}{t}dt+\varphi(1)<\infty$. 
In the last years, there have been several advances in the fruitful area of 
weighted inequalities concerning the precise determination of the optimal 
bounds of the weighted operator norm of linear and bilinear Calder\'on-Zygmund 
 operators with a Dini continuous kernel in terms of the $A_p$ constant 
 of weights (see e.g. \cite{DHL, Newyork} and the references therein). 
The algorithm to obtain sparse domination is formulated in \cite{Newyork} in 
general
and
can be used to study both standard Calder\'on-Zygmund operators and square 
functions.
However, in order to obtain estimates
for kernels satisfying the Dini condition, the main obstacle is 
the endpoint estimate and its bound. 
In fact, under the Dini assumptions, 
some weak type tricks have not been available until now. 

The thrust of relaxing the $\log$-Dini condition to the Dini condition
comes from the works \cite{HRT17,Lacey17}.
In fact, in these papers, the authors obtained
the sharp estimates
for singular integral operators whose kernels satisfy
the Dini condition.
Thus, it is natural to ask ourselves whether
a counterpart to the Littlewood--Paley operators
is available.

We would like to point out the difference 
between the proofs of the main results
of this paper and
\cite{SXY},
where
the authors in 
\cite{SXY}
assumed the $\log$-Dini condition.
To do so, we use the operators $S_{1,\psi}$ defined via the cone
(\ref{eq:220403-1111}) for $\alpha=1$
and $g^*_{\lambda,\psi}$ defined via rapidly decaying weights.
We refer to (\ref{eq:220411-2}) and (\ref{eq:220411-3})
below for the definition of the operators.
But at this moment let us content ourselves with
the inequality
$S_{1,\psi} \lesssim g^{*}_{\lambda, \psi}$.
Seemingly, it is insufficient to handle $S_{1,\psi}$ solely.
However,
as is seen from (\ref{eq:211107-1}) and so on,
we can recover the boundedness properties
of $S_{\alpha,\psi}$ for $\alpha>1$
from $S_{1,\psi}$.
Furthermore,
(\ref{grelationS-linear}) allows us to recover
the estimate for $g_{\lambda,\psi}^*$.
We will mainly consider
the property of $S_{1,\psi}$,
while
the authors in
\cite{SXY}
dealt with $g^*_{\lambda,\psi}$.

 Moreover, 
 we assume a much weaker condition a priori i.e. the boundedness of $S_{1,\psi}$,
while
the boundedness of $g^*_{\lambda,\psi}$ is assumed in
\cite{SXY}.

The main aim of this paper is, notwithstanding the generalization, 
to simplify and shorten the proofs of the main results of \cite{SXY} significantly 
so that we get a precise linear dependence on Dini constants.
In particular, 
our new proof for \cite[Theorem 1.1]{SXY} 
is quite short and, unlike the proof given in \cite{SXY}, does not rely on the properties of 
the ``Marcinkiewicz function''. 
By means of the product, 
we relax the log-Dini condition in the pointwise bound to the classical Dini condition. See
(\ref{eq-sizecondition}),
(\ref{eq-smoothcondition1})
and
(\ref{eq-smoothcondition2})
for more.
This solves an open problem (see e.g. \cite[pp. 37--38]{CY}). 

In this paper, we discuss separately the linear case and the multilinear 
case for the sake of clarity. That is, we treat the linear case carefully.
The multilinear case is sketchy
due to similarity. 

We employ the following notation:
\begin{itemize}
\item
Denote by $\sharp A$ the cardinality of the set $A$
and by $\overline{A}$ its closure.
\item
For $\kappa,x>0$, write $\log^{\kappa}x=(\log x)^{\kappa}$.
\item
 Denote by $v_n$ the volume of the unit ball.
\item
By a cube,
we mean
a compact or right-open cube whose edges are parallel to 
the
coordinate axes.
\item
Denote by $Q(x,r)$ the closed cube centered at $x$ of volume $(2r)^n$.
For a cube $Q$,
$c(Q)$ stands for its center,
while
$\ell(Q)=|Q|^{\frac1n}$.
Thus,
$c(Q(x,r))=x$ and $\ell(Q(x,r))=2r$.
\item 
For a right-open cube $Q=\prod\limits_{j=1}^n[a_{j},b_{j})$,
the set ${\mathcal D}_{1}(Q)$ of the children of $Q$
stands for the set of all right-open cubes obtained by bisecting $Q$.
Define inductively ${\mathcal D}_k(Q)$ by
$
{\mathcal D}_k(Q)
\equiv\bigcup\limits_{R \in {\mathcal D}_{k-1}(Q)}{\mathcal D}_{1}(R).
$
Denote by ${\mathcal D}(Q)$ the set of all dyadic cubes
generated by
a right-open cube $Q$.
\item
We denote by
${\mathcal D}$ the set of all dyadic cubes.
\item

For $k=1,2,3$,
let
${\mathcal D}^{(k)}({\mathbb R})$
be the minimal family
satisfying the following conditions:
\begin{itemize}
\item
$\{[3j+k-1,3j+k)\}_{j \in {\mathbb Z}} 
\subset {\mathcal D}^{(k)}({\mathbb R})$.
\item
If
$I_1 \in {\mathcal D}^{(k)}({\mathbb R})$
satisfies
$\ell(I_1)=2\ell(I_2)$
or
$\ell(I_2)=2\ell(I_1)$
and
$\sharp(\overline{I_1} \cap \overline{I_2})=1$,
then
$I_2 \in {\mathcal D}^{(k)}({\mathbb R})$.
\end{itemize}For
$\vec{k}=(k_1,k_2,\ldots,k_n) \in \{1,2,3\}^n$,
we define
\begin{equation}\label{eq:220411-1}
{\mathcal D}^{(\vec{k})}({\mathbb R}^n)
=
\{I_1 \times I_2 \times \cdots \times I_n\,:\,
I_j \in {\mathcal D}^{(k_j)}({\mathbb R}),
\ell(I_1)=\ell(I_2)=\cdots=\ell(I_n)\}.
\end{equation}
\item 
We use the symbol $\fint_Q$ to denote the integral average
over a cube $Q$.
\item
Let $Q$ be a cube.
For $f \in L^1(Q)$,
we define
\[
\langle f \rangle_{1,Q}:=
\fint_{Q}|f(x)|dx.
\]
\item
Denote by ${\mathcal M}$ the Hardy--Littlewood maximal operator.
The operator $M_{{\mathcal D}}$ is the maximal operator generated by
the family ${\mathcal D}$,
that is,
\[
M_{{\mathcal D}}f(x)=\sup_{Q \in {\mathcal D}}1_Q(x)
\langle f \rangle_{1,Q}.
\]
Let $\kappa>0$.
We also denote
by ${\mathcal M}_\kappa$ the powered maximal operator:
\[
{\mathcal M}_\kappa f(x)={\mathcal M}[|f
|^\kappa](x)^{\frac{1}{\kappa}}
\quad (x \in {\mathbb R}^n).
\]
\item
For $f \in L^1({\mathbb R}^n)$,
define
the Fourier transform
by{\rm:}
\begin{eqnarray*}
{\mathcal F}f(\xi)
\equiv
(2\pi)^{-\frac{n}{2}}
\int_{{\mathbb R}^n} f(x){\rm e}^{-{\rm i} x \cdot \xi} dx.
\end{eqnarray*}

\item
For an operator $S$,
we write $S^2$ for the operator given by $S^2f=(Sf)^2$.
\end{itemize}

We describe the organization of this paper.
Main results for linear operators
are stated in Section \ref{s1},
while
Section \ref{s7} seeks a passage to the multilinear case.
We illustrate by a couple of examples
that the $\log$-Dini condition is not necessary
for the sparse decomposition
in Section \ref{s10.11},
where we prove the main result for linear operators.
We end this paper with some related results in Section \ref{sect5}.
\section{The linear Littlewood--Paley operators}\label{s1}
Here we state the main results for linear operators.
We give the definition of linear Littlewood--Paley operators
in Section \ref{sect2.1}.
A historical remark is given after we define
the operators we consider in this paper
in Section \ref{sect2.1}.
We indicate how to use the Dini condition in Section \ref{s2}
for Littlewood--Paley operators
whose kernels satisfy the Dini condition.
Section \ref{sec2} is oriented 
to the weak $L^1$-boundedness of $S_{\alpha,\psi}$.
As a preparatory step to prove our theorem,
in Section \ref{secweak2}
we consider two auxiliary operators
${\mathcal M}_{S_{\alpha,\psi}}$ and 
${\mathcal N}_{S_{\alpha,\psi}}$.
Section \ref{sparse-lastpart} proves the main result
formulated in Section \ref{sect2.1}.
\subsection{Definition}\label{sect2.1}
Let us recall the definition of square functions considered in this paper.

\medskip
Let $\psi(x,y)$ be a 
real-valued locally integrable function defined away from the diagonal 
$x =y$ in $\mathbb{R}^{2n}$. 
Let $\varphi,\,w$ be moduli of continuity. We assume that there is a positive 
constant $ A$ so that the following conditions hold
for any $x,y,h \in {\mathbb R}^n$:
\begin{enumerate}[Size condition:]
 \item \begin{equation}\label{eq-sizecondition}
 |\psi(x,y)|\leq A{\mathcal M}1_{Q(x,1)}(y)w\Big( \f{1}{1+|x-y|} \Big).
 \end{equation}
 \end{enumerate}

\begin{enumerate}[Smoothness condition:]
 \item
Whenever $|h|<\f{1}{2}|x-y|$, 
 \begin{equation}\label{eq-smoothcondition1}
 |\psi(x,y)-\psi(x+h,y)|\leq
 A{\mathcal M}1_{Q(x,1)}(y)\varphi\Bigl(\frac{|h|}{1+|x-y|}\Bigr)
 w\Big( \f{1}{1+|x-y|} \Big)
 \end{equation}
 and
\begin{equation}\label{eq-smoothcondition2}
 |\psi(x,y)-\psi(x, y+h)|\leq A{\mathcal M}1_{Q(x,1)}(y)
 \varphi\Bigl(\frac{|h|}{1+|x-y|}\Bigr)w\Big( \f{1}{1+|x-y|} \Big).
\end{equation}
\end{enumerate}

For $t>0$ we define a linear operator $\psi_t$ by
$$
\psi_t f (x)=\f{1}{t^{n}}\int_{\mathbb{R}^n}\psi\Big(\f{x}{t},\f{y}{t}\Big) 
f(y)dy
$$
for $f\in {\mathcal S}(\mathbb{R}^n)$.
For $\lambda>2$
and $\alpha>0$, the square functions $g^*_{\lambda, \psi}$ and
 $S_{\psi,\alpha}$ associated to $\psi$ are defined by
\begin{equation}\label{eq:220411-2}
g^*_{\lambda, \psi} f (x):=\Big(\iint_{{\mathbb R}^{n+1}_+}
\Big(\f{t}{t+|x-y|}\Big)^{n\lambda}|\psi_t f (y)|^2
\f{dydt}{t^{n+1}}\Big)^{\frac12} \quad (x \in {\mathbb R}^n)\end{equation}
and
\begin{equation}\label{eq:220411-3}
S_{\alpha, \psi} f (x):=\Big(\iint_{\Gamma_\alpha(x)}|\psi_t f (y)|^2
\f{dydt}{t^{n+1}}\Big)^{\frac12} \quad (x \in {\mathbb R}^n),
\end{equation}
respectively,
where $\Gamma_\alpha(x)$
is given by (\ref{eq:220403-1111}).
We note that these operators
are generalizations and expansions
of the operators handled 
in \cite[Section 7, (7.2)--(7.4)]{LLO}.
In comparison with the paper \cite{LLO},
we can consider the continuous wavelet expansions.
 We assume a priori
 in this paper that $S_{1,\psi}$ is $L^2$-bounded. 
A direct consequence
of this assumption
is that $S_{\alpha,\psi}$ is $L^2$-bounded for any 
$\alpha>0$, 
as is seen from $\|S_{\alpha,\psi}f\|_{L^2}
=\alpha^{n/2}\|S_{1,\psi}f\|_{L^2}$ 
(see Lemma \ref{l2boundedness}). 
Also, using \cite[Lemma 2.1]{Ler5} and \cite[Lemma 3.1]{BuiHormozi}, 
we have 
\begin{equation}\label{eq:211107-1}
\|S_{\alpha,\psi}\|_{L^1 \to L^{1,\infty}} 
\lesssim \alpha^{n} \|S_{1,\psi}\|_{L^1 \to L^{1,\infty}}.
\end{equation}

The starting point in this section is the weak-$(1,1)$ estimate
of $S_{\alpha,\psi}$.
\begin{theorem}\label{prop:201226-2}
If $\varphi$ and $w$ are moduli of continuity,
then there exist $A>0$ and $D_1>0$ such that
$\|S_{\alpha,\psi}\|_{L^1 \to L^{1,\infty}} \le D_1 \alpha^n
(A [w]_{\rm{Dini}}(1+[\varphi]_{\rm{Dini}})+\|S_{1,\psi}\|_{L^2 \to L^2})$
for all $\alpha \ge 1$.
\end{theorem}
As is seen from
\begin{equation}\label{grelationS-linear}
 g_{\lambda, \psi}^{*}f
 \lesssim \sum_{k=0}^{\infty} 2^{-\frac{k\lambda n} {2}} S_{ 2^{k+1}, \psi}f
 \quad (f \in L^1({\mathbb R}^n)),
\end{equation} 
Theorem \ref{prop:201226-2} guarantees the boundedness of 
$g_{\lambda,\psi}^*$ for $\lambda>2$.

Let
${\mathfrak D}$ be either ${\mathcal D}({\mathbb R}^n)$
or a family given by (\ref{eq:220411-1}).
Let $0<\eta<1$.
We say that
a family ${\mathcal S}$ is said to be $\eta$-sparse,
if
\[
\left|\bigcup_{R \in {\mathcal S}, R \subsetneq Q}R\right|
\le (1-\eta)|Q|.
\]
If $\eta=\frac12$,
then we simply say that
a family ${\mathcal S}$ is sparse.
In this paper, to save the number of parameters,
we let $\eta=\frac12$.
However, a slight modification allows us to extend the case
of $\eta\in(0,\frac12)$.
We are now ready to state our main result.
\begin{theorem}\label{domiCor1}
For any $\alpha \geq 1$
and for all compactly supported $f \in L^1(\Rn)$,
there exist sparse families
${\mathcal S}^{(\vec{k})} \subset {\mathcal D}^{(\vec{k})}$ 
$($depending on $f$$)$
for each
$\vec{k}=(k_1,k_2,\ldots,k_n) \in \{1,2,3\}^n$ such that
\[
S_{\alpha,\psi} f\cdot 1_{Q_0}
\lesssim \alpha^n
\log^{\frac32}(2+\alpha)
([\varphi]_{\rm{Dini}} [w]_{\rm{Dini}}+\|S_{1,\psi}\|_{L^2 \to L^2})
\sum_{\vec{k} \in \{1,2,3\}^n}
\Big[\sum_{P \in {\mathcal S}^{(\vec{k})}} \langle f \rangle_{1,P}^2 1_P\Big]^{\frac{1}{2}}.
\]
Here the implicit constant in $\lesssim$ is independent of $\alpha$,
$[w]_{\rm Dini}$
and
$[\varphi]_{\rm Dini}$.
\end{theorem}
A standard argument as in \cite{Newyork}
boils down the proof of Theorem \ref{domiCor1}
to the following lemma:
\begin{lemma}\label{domiTheorem}
For any $\alpha \geq 1$
and 
for any $f \in L^1(\Rn)$ 
supported in $3Q_0$ for some $Q_0 \in \mathscr{D}$,
there exists a sparse family ${\mathcal S} \subset {\mathcal D}(Q_0)$ 
$($depending on $f$$)$ such that
\begin{equation}\label{Sparse-Dini}
S_{\alpha,\psi} f\cdot 1_{Q_0}
\lesssim \alpha^n
\log^{\frac32}(2+\alpha)
([\varphi]_{\rm{Dini}} [w]_{\rm{Dini}}+\|S_{1,\psi}\|_{L^2 \to L^2})
\Big[\sum_{P \in {\mathcal S}} 
\langle f \rangle_{1,3P}^2 1_P\Big]^{\frac{1}{2}}.
\end{equation}
Here the implicit constant is
independent of $\alpha$,
$\varphi$ and $w$.
\end{lemma}
We offer words to the existing results.
First of all,
Lemma \ref{domiTheorem} resembles
\cite[Theorem 7.2]{LLO}.
However,
Lemma \ref{domiTheorem} carefully keeps track
of the constant.
So, we will describe its proof from scratch.
Let us also remark that a passage to the multilinear case is done
later using the same idea in the linear case.
By considering $w(\xi)=\xi^\delta$ and $\varphi(\xi)=\xi^\gamma$, 
we learn that
our result is a generalization of \cite[Theorem 1.1]{SXY}. Notwithstanding the generalization, our proof for 
 \cite[Theorem 1.1]{SXY} is simpler and, unlike the proof given 
 in \cite{SXY}, does not rely on the properties of 
 the ``Marcinkiewicz function''.
 
In the sequel, for the sake of later use, we define 
$\varphi(t):=\varphi(1)$ and 
$w(t):=w(1)$ for $t>1$, without loss of generality. 
We note that for any $A>1$, 
$\int_0^{A}\varphi(t)dt/t\le C_A [\varphi]_{\rm{Dini}}$ and 
$\int_0^{A}w(t)dt/t\le C_A [w]_{\rm{Dini}}$. 
The size of $C_A$ is important because
it contributes to the power of $\log(2+\alpha)$ in 
(\ref{Sparse-Dini}).
See Lemma \ref{auxiliary}
for how fast $C_A$ grows.
It is noteworthy not to assume the $\log$-Dini condition
on $w$ and $\varphi$, that is,
we do not have to assume that
\[
\int_0^1 w(t)\log \frac1t \cdot \frac{dt}{t}<\infty
\]
or that
\[
\int_0^1 \varphi(t)\log \frac1t \cdot \frac{dt}{t}<\infty.
\]

Motivated by the works of Coifman and Meyer \cite{CM},
Shi, Xue and Yan
together with the third author introduced and investigated
the multilinear square functions given by
(\ref{eq:220411-2})
and
(\ref{eq:220411-3}) in \cite{SXY, XY}. 
The multilinear square functions have 
important applications in PDEs and other fields
of mathematics. 
In particular, Fabes, Jerison, and Kenig applied multilinear square functions in PDEs. 
For example,
in \cite{FJK1}, they studied the solutions of
the Cauchy problem for non-divergence form parabolic equations 
via some multilinear Littlewood--Paley type estimates for the square root of an elliptic operator in divergence form. 
Also,
they obtained the necessary and sufficient conditions for absolute continuity of elliptic-harmonic measure
using
a multilinear Littlewood--Paley estimate in \cite{FJK2}. 
Moreover, in \cite{FJK3}, they applied a class of multilinear square functions to Kato's problem. 
For further details on the theory of multilinear square functions and their applications, 
we refer to \cite{CY, CDM, CMM, CM, FJK1, FJK3} and the references therein.
 
\subsection{Auxiliary estimates}\label{s2}
First, for the case of $L^2(\Rn)$, we verify how strongly
the operator norm depends on
the aperture $\alpha$.
\begin{lemma}\label{l2boundedness}
Let $\alpha \geq 1$. Then
$\|S_{\alpha,\psi}f\|_{L^2}
=\alpha^{n/2}\|S_{1,\psi}f\|_{L^2} $
for all $f \in L^2({\mathbb R}^n)$.
\end{lemma}
 
 \begin{proof}
 Let $f \in L^2({\mathbb R}^n)$.
 Then
 \begin{align*}
\|S_{\alpha,\psi}f\|_{L^2}^2 
&= \int_{\Rn} \left( \iint_{\Gamma_{\alpha(x)}} |\psi_tf(y)|^2 \frac{dy dt}{t^{n+1}} \right) dx \\
&=\iint_{{\mathbb R}^{n+1}_+} |\psi_tf(y)|^2 \left( \int_{\Rn} 1_{\{|y-x|<\alpha t \}}(x) dx \right) \frac{dydt}{t^{n+1}} \\
&= v_n \alpha^n \iint_{{\mathbb R}^{n+1}_+} |\psi_tf(y)|^2 \frac{dydt}{t}. 
\end{align*}
If we let $\alpha=1$ in the above, then we obtain
\[
\|S_{1,\psi}f\|_{L^2}^2 
=v_n \iint_{{\mathbb R}^{n+1}_+} |\psi_tf(y)|^2 \frac{dydt}{t}.
\]
From these two equalities, we obtain the desired result.
\end{proof}

We will use the following lemma several times:
\begin{lemma}\label{auxiliary} Let $\alpha \geq 1$. 
Let $w$ be a modulus of continuity satisfying the Dini condition.
\begin{enumerate}
\item[$(a)$]
$\displaystyle
\left(
\int_0^\infty
\left|
\frac{1}{(1+t)^{n}}
w\left(\frac{4t}{t+1}\right)
\right|^2\frac{dt}{t}
\right)^{\frac12} \lesssim [w]_{\rm{Dini}}.
$
\item[$(b)$]
$\displaystyle
\int_0^\infty
\left\{
\frac{1}{(t+1)^{n}}
w\left((1+\alpha)t\right)\right\}^2\frac{dt}{t} \lesssim 
\log (2+\alpha)[w]_{\rm{Dini}}^2.
$
\item[$(c)$] 
$\displaystyle\sum_{k=1}^{\infty} w\left( \frac{1+\alpha}{2^{k+1}} \right) \lesssim \log(2+\alpha) [w]_{\rm{Dini}}.
$
\item[$(d)$]
$\displaystyle
\int_0^\alpha w(t)\frac{dt}{t} \lesssim \log(2+\alpha)[w]_{\rm{Dini}}.
$
\item[$(e)$]Let $m \in \mathbb{N}$. Then
$[{w}]_{\rm{Dini}}\le[w ^{\frac{1}{m}}]^{m}_{\rm{Dini}}$.
\end{enumerate}

\end{lemma}

\begin{proof}\
\begin{enumerate}
\item[$(a)$] 
Simply change variables $u=\frac{4t}{t+1}$.
We omit further details.
\item[$(b)$]
Write $I_1'$ for the integral on the left-hand side.
We decompose the integral defining $I^{'}_{1}$ into three parts.
\begin{align*}
I^{'}_{1}&=\int_{0}^{\frac{1}{1+\alpha}} \cdots +\int_{\frac{1}{1+\alpha}} ^1 \cdots+ \int_{1}^{\infty} \cdots\\
&\leq w(1) \int_{0}^{\frac{1}{1+\alpha}} \frac{w\left((2+\alpha)t\right)}{t} dt 
+ w^2(1) \int_{\frac{1}{1+\alpha}} ^1 \frac{dt}{t} + w^2(1) \int_{1}^{\infty} \frac{dt}{(t+1)^{2n}}.\\
&\leq w(1) [w]_{\rm{Dini}} + w^2(1) \log (2+\alpha) + w^2(1)\\
&\lesssim \log (2+\alpha)[w]_{\rm{Dini}}^2.
\end{align*}
\item[$(c)$] 
Write $I_2'$ for the sum on the left-hand side.
By the integral test, we have
 \begin{align*}
 \sum_{k=1}^{\infty} w\left( \frac{1+\alpha}{2^{k+1}} \right) &\leq w(1)\log(2+\alpha) + \int_{1}^{\infty} 
 w\left( \frac{1+\alpha}{2^{x}} \right) dx\\
 &\lesssim w(1)\log(2+\alpha)+ \int_{0}^{ \frac{1+\alpha}{2}} \frac{w(t)}{t} dt.\\
 &\lesssim \log(2+\alpha)[w]_{\rm{Dini}}+ \int_{0}^{ 1}\frac{w(t)}{t} dt + w(1) \int_{1}^{\alpha} \frac{1}{t} dt. \\
 &\lesssim \log(2+\alpha)[w]_{\rm{Dini}}.
 \end{align*}
 \item[$(d)$]
 Argue similar to $(c)$.
 \item[$(e)$]
 We calculate
$$
\int_{0}^{1} \frac{w(t)}{t}dt 
\leq w^{\frac{m-1}{m}}(1) \int_{0}^{1} \frac{w^{\frac1m}(t)}{t}dt 
\leq [w ^{\frac{1}{m}}]^{m-1}_{\rm{Dini}} [w ^{\frac{1}{m}}]_{\rm{Dini}}.
$$
 \end{enumerate}
 \end{proof}

The above estimates concern
how to use the Dini condition
for a $1$-dimensional integral.
Now we show how to use
it for the integral over ${\mathbb R}^n$.
For a cube $Q$,
we denote by $c(Q)$ its center and
by $\ell(Q)$ its side-length.
\begin{lemma}\label{lem:210801-1}
Let $w$ be the modulus of continuity satisfying the Dini condition.
Then
for any fixed $m\in \mathbb N$, we have
\[
\sum_{k=1}^\infty
2^{-k n/2}
\int_{{\mathbb R}^n}
\frac{1}{(2^k\ell(Q)+|x-c(Q)|)^n}
w\left(\frac{2^{k+m}\ell(Q)}{2^k\ell(Q)+|x-c(Q)|}\right)dx\lesssim
[w]_{\rm Dini}.
\]

\end{lemma}
\begin{proof}
Fix $k \in {\mathbb N}$.
We calculate
\begin{align*}
\lefteqn{
\int_{{\mathbb R}^n}
\frac{1}{(2^k\ell(Q)+|x-c(Q)|)^n}
w\left(\frac{2^{k+m}\ell(Q)}{2^k\ell(Q)+|x-c(Q)|}\right)dx
}\\
&=
\int_{2^k Q}
\frac{1}{(2^k\ell(Q)+|x-c(Q)|)^n}
w\left(\frac{2^{k+m}\ell(Q)}{2^k\ell(Q)+|x-c(Q)|}\right)dx\\
&\quad+
\sum_{l=1}^\infty
\int_{2^{k+l}Q \setminus 2^{k+l-1}Q}
\frac{1}{(2^k\ell(Q)+|x-c(Q)|)^n}
w\left(\frac{2^{k+m}\ell(Q)}{2^k\ell(Q)+|x-c(Q)|}\right)dx\\
&\le
\int_{2^k Q}
\frac{1}{(2^k\ell(Q))^n}w(2^m)dx+
\sum_{l=1}^\infty
\int_{2^{k+l}Q \setminus 2^{k+l-1}Q}
\frac{1}{(2^{k+l-2}\ell(Q))^n}
w\left(\frac{2^{k+m}\ell(Q)}{2^{k+l-2}\ell(Q)}\right)dx
\\
&\lesssim
w(2^m)+\sum_{l=1}^{\infty}w(2^{m-l+1})
\\
&\lesssim
[w]_{\rm Dini}.
\end{align*}
Thus, if we add this estimate over $k$
after we multiply it by $2^{-kn/2}$,
then we obtain the desired result.
\end{proof}
Similar to Lemma \ref{lem:210801-1},
we use the following estimate:
\begin{cor}\label{cor:220330-1}
Let $w$ be the modulus of continuity satisfying the Dini condition.
Then
for any cube $Q$, 
\begin{equation}\label{eq:211112-4}
\int_{|x-c(Q)|>32n\ell(Q)}
\frac{1}{|x-c(Q)|^{n}}
w\left(\frac{2\sqrt{n}\ell(Q)}{|x-c(Q)|}\right)dx
\lesssim
[w]_{\rm{Dini}}.
\end{equation}
\end{cor}

\begin{proof}
Similar to Lemma \ref{lem:210801-1},
we calculate
\begin{align*}
&\int_{|x-c(Q)|>32n\ell(Q)}
\frac{1}{|x-c(Q)|^{n}}
w\left(\frac{2\sqrt{n}\ell(Q)}{|x-c(Q)|}\right)dx
\\
&\sim\int_{32n\ell(Q)}^{\infty}\frac{w(2\sqrt{n}\ell(Q)/r)}{r}dr
\sim \int_{0}^{2/(3\sqrt{n})}{w(s)}\frac{ds}{s}
\le
[w]_{\rm{Dini}}.
\end{align*}
\end{proof}

We indicate how to use the above estimates
for Littlewood--Paley operators.
\begin{lemma}\label{lem:220411-1}Let $Q$ be a cube and let $f$ be an integrable function supported
on the cube $Q$.
Assume that
the moment condition $\int\limits_Q f(x)dx=0$
holds.
Then
for all $x \in {\mathbb R}^n$,
\begin{align}
\label{eq:220328-1}
&\lefteqn{
1_{(64n\ell(Q),\infty)}(|x-c(Q)|)S_{1,\psi}f(x)
}\\
\nonumber
&\lesssim
\frac{A\,[w]_{\rm{Dini}}}{|x-c(Q)|^n}
\varphi\left(\frac{2\sqrt{n}\ell(Q)}{|x-c(Q)|}\right)\|f\|_{L^1}
\\
\nonumber
&\hspace{1.5cm}+
\sum_{k=1}^\infty
\frac{A\,2^{-k n/2}}{(2^k\ell(Q)+|x-c(Q)|)^n}
w\left(\frac{2^{k+2}\ell(Q)}{2^k\ell(Q)+|x-c(Q)|}\right)
{\|f\|_{L^1}}.
\end{align}
\end{lemma}

\begin{proof}
We seek to estimate
\[
\left(
\iint_{\Gamma_1(0)}
\left|
\frac{1}{t^n}
\int_{{\mathbb R}^n}\psi\left(\frac{x+z}{t},\frac{y}{t}\right)f(y)dy
\right|^2\frac{dzdt}{t^{n+1}}
\right)^{\frac12}
.\]
We start with a preparatory geometric observation:
Let $y\in Q$, $(z,t) \in \Gamma_1(0)$.
We use the following geometric observation:
\begin{lemma}
Let $Q$ be a cube and $(z,t) \in \Gamma_1(0)$.
Then
for all $y \in Q$ and $x \in {\mathbb R}^n$
satisfying $|x-c(Q)|>4n \ell(Q)$, we have
\begin{equation}\label{eq:220330-201}
4( t+ |x+z-y|)
> t+ |x-c(Q)|+\ell(Q).
\end{equation}
\end{lemma}
\begin{proof}
Since $t \ge|z|$,
\begin{align*}
4( t+ |x+z-y|)
&\ge
4t+2|x+z-y|\\
&\ge
2t+2|x-y|\\
&\ge
t+|x-c(Q)|+\ell(Q)+(|x-c(Q)|-2|y-c(Q)|-\ell(Q)).
\end{align*}
Thus, by assumption,
we obtain
(\ref{eq:220330-201}).
\end{proof}
We will decompose the integral defining
$S_{1,\psi}f(x)$ into two parts:
\begin{enumerate}
\item
Assume first that
\[
32n\ell(Q)>|x+z-c(Q)|.
\]
Let $y\in Q$, $|x-c(Q)|>64n\ell(Q)$, and let $|z|\le t$. 
Then 
a geometric observation shows that
\[
64n\ell(Q) <|x-c(Q)| 
\leq |x+z-c(Q)|+|z| < 32n \ell(Q) + t
\]
for all $(z,t) \in \Gamma_1(0)$.
Consequently, we have
\begin{equation}\label{eq:220402-1}
t> 32n \ell(Q).
\end{equation} 
By using 
(\ref{eq-sizecondition})
and
(\ref{eq:220330-201}),
we estimate
\begin{align*}
\left|
\frac{1}{t^n}
\int_{{\mathbb R}^n}
\psi\left(\frac{x+z}{t},\frac{y}{t}\right)f(y)dy
\right|
&\lesssim
\int_{{\mathbb R}^n}
\frac{1}{(t+|x+z-y|)^n}
w\left(\frac{t}{t+|x+z-y|}\right)|f(y)|dy\\
&\lesssim
\int_{Q}
\frac{1}{(t+|x-c(Q)|)^n}
w\left(\frac{4t}{t+|x-c(Q)|}\right)|f(y)|dy.
\end{align*}
Hence,
by applying Minkowski's inequality
to the above estimate
and then using (\ref{eq:220402-1}),
we have
\begin{align*}
\lefteqn{
\left(
\iint_{(z,t) \in \Gamma_1(0), |x+z-c(Q)|< 32n\ell(Q)}
\left|
\frac{1}{t^n}
\int_{{\mathbb R}^n}\psi\left(\frac{x+z}{t},\frac{y}{t}\right)f(y)dy
\right|^2\frac{dzdt}{t^{n+1}}
\right)^{\frac12}
}
\\
&\lesssim
\|f\|_{L^1}
\left(
\iint_{(z,t) \in \Gamma_1(0)}
\left|
\frac{1_{\{z: |x+z-\ell(Q) | < 32 n \ell(Q) \} \times [\ell(Q), \infty)}(z,t)}{(t+|x-c(Q)|)^n}
w\left(\frac{4t}{t+|x-c(Q)|}\right)
\right|^2\frac{dzdt}{t^{n+1}}
\right)^{\frac12}
\\
&\sim
\|f\|_{L^1}
\left(
\int_{\ell(Q)}^\infty\ell(Q)^{n}
\left|
\frac{1}{(t+|x-c(Q)|)^n}
w\left(\frac{4t}{t+|x-c(Q)|}\right)
\right|^2\frac{dt}{t^{n+1}}
\right)^{\frac12}.
\end{align*}
By using
the inequality 
\[
\sqrt{A+B} \le \sqrt{A}+\sqrt{B}
\] for $A,B \ge 0$,
we have
\begin{align*}
\lefteqn{
\left(
\iint_{(z,t) \in \Gamma_1(0), |x+z-c(Q)|< 32n\ell(Q)}
\left|
\frac{1}{t^n}
\int_{{\mathbb R}^n}\psi\left(\frac{x+z}{t},\frac{y}{t}\right)f(y)dy
\right|^2\frac{dzdt}{t^{n+1}}
\right)^{\frac12}
}
\\
&\lesssim
\|f\|_{L^1}
\sum_{k=1}^\infty
2^{-k n/2}
\frac{1}{(2^k\ell(Q)+|x-c(Q)|)^n}
w\left(\frac{2^{k+2}\ell(Q)}{2^k\ell(Q)+|x-c(Q)|}\right).
\end{align*}
Thus,
the estimate for $x \in {\mathbb R}^n$ satisfying
$|x+z-c(Q)|<32n\ell(Q)$ is valid;
the most right-hand side we obtained
matches
the second term of the right-hand side of 
(\ref{eq:220328-1}).
\item
Assume instead $32n\ell(Q)<|x+z-c(Q)|$. 
Let $y\in Q$, $|x-c(Q)|>64n\ell(Q)$ and $(z,t) \in \Gamma_1(0)$.
We estimate by using
$|y-c(Q)| \le 2\sqrt{n}\ell(Q)$,
$y \in Q$,
 \eqref{eq-smoothcondition2}
and (\ref{eq:220330-201})
as well as the moment condition on $f$,
\begin{align*}
\lefteqn{
\left|
\frac{1}{t^n}
\int_{{\mathbb R}^n}
\psi\left(\frac{x+z}{t},\frac{y}{t}\right)f(y)dy
\right|
}
\\
&=
\left|
\frac{1}{t^n}
\int_{{\mathbb R}^n}
\psi\left(\frac{x+z}{t},\frac{y}{t}\right)f(y)dy-
\frac{1}{t^n}
\int_{{\mathbb R}^n}
\psi\left(\frac{x+z}{t},\frac{c(Q)}{t}\right)f(y)dy
\right|
\\
&\lesssim
\int_{{\mathbb R}^n}
\frac{1}{(t+|x+z-y|)^{n}}\varphi\left(\frac{|y-c(Q)|}{t+|x+z-y|}\right)
w\left(\frac{t}{t+|x+z-y|}\right)|f(y)|dy
\\
&\lesssim
\int_{{\mathbb R}^n}
\frac{1}{(t+|x-c(Q)|)^{n}}
\varphi\left(\frac{2\sqrt{n}\ell(Q)}{|x-c(Q)|}\right)
w\left(\frac{4t}{t+|x-c(Q)|}\right)|f(y)|dy.
\end{align*}
If we combine this estimate
with Minkowski's inequality,
\begin{align*}
\lefteqn{
\left(
\iint_{(z,t) \in \Gamma_1(0), |x+z-c(Q)|\ge 32n\ell(Q)}
\left|
\frac{1}{t^n}
\int_{{\mathbb R}^n} 
\psi\left(\frac{x+z}{t},\frac{y}{t}\right)f(y)dy
\right|^2\frac{dzdt}{t^{n+1}}
\right)^{\frac12}
}
\\
&\lesssim
\left(
\iint_{\Gamma_1(0)}
\left|
\frac{1}{(t+|x-c(Q)|)^{n}}
\varphi\left(\frac{2\sqrt{n}\ell(Q)}{|x-c(Q)|}\right)
w\left(\frac{4t}{t+|x-c(Q)|}\right)
\right|^2\frac{dzdt}{t^{n+1}}
\right)^{\frac12}
\|f\|_{L^1}
\\
&\sim
\left(
\int_0^\infty
\left|
\frac{1}{(t+|x-c(Q)|)^{n}}
\varphi\left(\frac{2\sqrt{n}\ell(Q)}{t+|x-c(Q)|}\right)
w\left(\frac{4t}{t+|x-c(Q)|}\right)
\right|^2\frac{dt}{t}
\right)^{\frac12}
\|f\|_{L^1}
\\
&\lesssim
\frac{1}{|x-c(Q)|^{n}}
\varphi\left(\frac{2\sqrt{n}\ell(Q)}{|x-c(Q)|}\right)
\left(
\int_0^\infty
\left|
\frac{1}{(1+t)^{n}}
w\left(\frac{4t}{t+1}\right)
\right|^2\frac{dt}{t}
\right)^{\frac12}
\|f\|_{L^1}.
\end{align*}
By Lemma \ref{auxiliary} (a)
and
(\ref{eq:211112-4}),
\begin{eqnarray*}
&&
\left(
\iint_{(z,t) \in \Gamma_1(0), |x+z-c(Q)|\ge 32n\ell(Q)}
\left|
\frac{1}{t^n}
\int_{{\mathbb R}^n} 
\psi\left(\frac{x+z}{t},\frac{y}{t}\right)f(y)dy
\right|^2\frac{dzdt}{t^{n+1}}
\right)^{\frac12}
\\
&&\lesssim
\frac{[w]_{\rm{Dini}}}{|x-c(Q)|^{n}}
\varphi\left(\frac{2\sqrt{n}\ell(Q)}{|x-c(Q)|}\right)
\|f\|_{L^1}.
\end{eqnarray*}
Thus, the estimate for $x \in {\mathbb R}^n$ satisfying $|x+z-c(Q)|\ge 32n\ell(Q)$ is valid;
the most right-hand side we obtained
matches the first term on the right-hand side of 
(\ref{eq:220328-1}).
\end{enumerate}
Thus, we get the desired estimate 
\eqref{eq:220328-1}.
\end{proof}

\subsection{Weak $L^1$-boundedness of $S_{\alpha,\psi}$--Proof of Theorem \ref{prop:201226-2}}\label{sec2}

Let $\rho>0$ and $f \in L^1({\mathbb R}^n)$ be fixed.
We let $E_\rho=\{x \in {\mathbb R}^n\,:\,|f(x)|>\rho\}$.
We write 
Let $\{Q_j\}_{j \in J}$ be a family of bounded-overlapping
(maximal) cubes
that decomposes 
$\{x \in {\mathbb R}^n\,:\,M_{{\mathcal D}}f(x)>\rho\}$
with the property that
\[
\rho \le \langle f \rangle_{1,Q_j} \le 2^{n}\rho
\quad (j \in J).
\]
For details,
see \cite{Duoandikoetxea01}, for example.

 

We prove Theorem \ref{prop:201226-2}.
We may assume $\alpha=1$
in view of (\ref{eq:211107-1}).
Form the Calder\'{o}n-Zygmund decomposition
of $f$ at height $\rho$.
Denote by $g$ the good part and set $b=\sum\limits_{j \in J} b_j=f-g$.
Here each $b_j$, $j \in J$ is given by
\[
b_j:=1_{Q_j}\left(f-\fint_{Q_j}f(y)dy\right).
\]
We have
\begin{equation}\label{eq:220330-1}
\rho|\{x \in {\mathbb R}^n\,:\,|S_{1,\psi}(g)(x)|>\rho\}|
\lesssim
\|S_{1,\psi}\|_{L^2 \to L^2}^2
\|f\|_{L^1},
\end{equation}
since $S_{1,\psi}$ is assumed to be $L^2$-bounded.

Set $\Omega:=\bigcup\limits_{j \in J} 64nQ_j$.
Remark that
\begin{equation}\label{eq:220406-1}
|\Omega| \lesssim \frac{1}{\rho}\|f\|_{L^1}.
\end{equation}
Thanks to Lemmas \ref{lem:220411-1}
and \ref{lem:210801-1}
and 
Corollary \ref{cor:220330-1},
\begin{equation}\label{eq:weak-estimate-10}
\int_{{\mathbb R}^n \setminus \Omega}S_{1,\psi}b_j(x)dx
\lesssim[w]_{\rm Dini}([\varphi]_{\rm Dini}+1)
\|b_j\|_{L^1}\lesssim[w]_{\rm Dini}([\varphi]_{\rm Dini}+1)
\|f\|_{L^1(Q_j)}
\end{equation}
for each $j \in J$.
Estimate (\ref{eq:weak-estimate-10})
will complete the proof, since
(\ref{eq:weak-estimate-10})
yields the control of $b$, while
(\ref{eq:220330-1}) takes care of $g$.
Thus, along with (\ref{eq:220406-1}), we obtain
\[
\rho|\{x \in {\mathbb R}^n\,:\,
S_{1,\psi}f(x)>\rho\}|
\lesssim
(\|S_{1,\psi}\|_{L^2 \to L^2}^{2}
+
1
+A[w]_{\rm Dini}([\varphi]_{\rm Dini}+1))\|f\|_{L^1},
\]
proving Theorem \ref{prop:201226-2}.


\subsection{Weak $L^1$-boundedness of ${\mathcal M}_{S_{\alpha,\psi}}$ and 
${\mathcal N}_{S_{\alpha,\psi}}$}\label{secweak2}

This section is an auxiliary step,
where we deal with Lerner's maximal function
\cite{Newyork}.
We modify its definition so that it is adapted to the $\ell^2$-valued case.
Let $x \in {\mathbb R}^n$.
We consider
\[
{\mathcal M}_{S_{\alpha,\psi}}f(x)
=
\sup_{Q}1_Q(x)
\sqrt{|S_{\alpha,\psi}f(x)^2-S_{\alpha,\psi}(f \cdot 1_{3Q})(x)^2|}
\quad (x \in {\mathbb R}^n)
\]
and
\[
{\mathcal N}_{S_{\alpha,\psi}}f(x)
=
\sup_{Q}1_Q(x)
S_{\alpha,\psi}(f \cdot 1_{{\mathbb R}^n \setminus 3Q})(x)
\quad (x \in {\mathbb R}^n),
\]
where $Q$ moves over all cubes. 
We will invoke
Kolmogorov's inequality.
See \cite[Lemma 5.16]{Duoandikoetxea01} for the proof.
\begin{lemma}\label{lem:Kol}
Let $S$ be a sublinear or linear operator bounded 
from $L^1({\mathbb R}^n)$ to $L^{1,\infty}({\mathbb R}^n)$.
Then for any set $E$ of finite measure,
we have\[
\int_E \sqrt{|S f(x)|}dx
\lesssim
\sqrt{\|S\|_{L^1 \to L^{1,\infty}}|E|\|f\|_{L^1}}.
\]
\end{lemma}

We seek to show:
\begin{lemma}\label{weakestiamte} 
There exists a constant $D_2>0$,
independent of $\alpha\ge 1$, such that
$$
\|{\mathcal N}_{S_{\alpha,\psi}}\|_{L^1 \to L^{1,\infty}}+
\|{\mathcal M}_{S_{\alpha,\psi}}\|_{L^1 \to L^{1,\infty}}
\le
D_2 \alpha^n\log^{\frac32}(2+\alpha)
([w]_{\rm{Dini}}[\varphi]_{\rm{Dini}}+\|S_{1,\psi}\|_{L^2 \to L^2}).
$$
\end{lemma}
\begin{proof}
It suffices to handle
${\mathcal N}_{S_{\alpha,\psi}}$.
In fact, arithmetic shows
\[
|a^2-b^2|=|a-b|\cdot|a+b|
\le|a-b|^2+2|a-b|\cdot|a|
\]
and hence
\[
{\mathcal M}_{S_{\alpha,\psi}}f(x)
\le
2\sqrt{{\mathcal N}_{S_{\alpha,\psi}}f(x)}
\sqrt{{\mathcal N}_{S_{\alpha,\psi}}f(x)+S_{\alpha,\psi} f (x)}
\quad (x \in {\mathbb R}^n).
\]
Thus, the estimate for
${\mathcal M}_{S_{\alpha,\psi}}$
follows
once we show that
${\mathcal N}_{S_{\alpha,\psi}}$ is weak-$(1,1)$ bounded
with the norm estimate
$$
\|{\mathcal N}_{S_{\alpha,\psi}}\|_{L^1 \to L^{1,\infty}}
\lesssim
[w]_{\rm{Dini}}[\varphi]_{\rm{Dini}}+\|S_{1,\psi}\|_{L^2 \to L^2}.
$$
Therefore,
We will seek a pointwise estimate
${\mathcal M}_{S_{\alpha,\psi}}$.
We fix a point $x \in {\mathbb R}^n$.

Fix a cube $Q$ so that $x \in Q$.
Let $x' \in Q$ be arbitrary.
We abbreviate
\[
{\rm I}=
|
S_{\alpha,\psi}(f \cdot 1_{{\mathbb R}^n \setminus 3Q})(x)
-
S_{\alpha,\psi}(f \cdot 1_{{\mathbb R}^n \setminus 3Q})(x')
|,
\]
so that 
\begin{equation}\label{eq:211114-1}
|
S_{\alpha,\psi}(f \cdot 1_{{\mathbb R}^n \setminus 3Q})(x)
|
\le
{\rm I}
+
|
S_{\alpha,\psi}(f \cdot 1_{{\mathbb R}^n \setminus 3Q})(x')
|.
\end{equation}
We claim that
\begin{equation}\label{eq:220330-3}
{\rm I}\lesssim [w]_{\rm{Dini}}{\mathcal M}f(x).
\end{equation}
Once (\ref{eq:220330-3}) is established,
we obtain
\begin{equation}\label{eq:220330-4}
|
S_{\alpha,\psi}(f \cdot 1_{{\mathbb R}^n \setminus 3Q})(x)
|
\lesssim [w]_{\rm{Dini}}{\mathcal M}f(x)+
|
S_{\alpha,\psi}(f \cdot 1_{{\mathbb R}^n \setminus 3Q})(x')
|
\end{equation}
for any $x' \in Q$
along with (\ref{eq:211114-1}). 

We write out $
S_{\alpha,\psi}(f \cdot 1_{{\mathbb R}^n \setminus 3Q})(x)$
and
$
S_{\alpha,\psi}(f \cdot 1_{{\mathbb R}^n \setminus 3Q})(x')$
in full:
\begin{align*}
S_{\alpha,\psi}(f \cdot 1_{{\mathbb R}^n \setminus 3Q})(x)
&=
\left(
\iint_{\Gamma_\alpha(0)}
|\psi_t(f \cdot 1_{{\mathbb R}^n \setminus 3Q})(x+y)|^2\frac{dydt}{t^{n+1}}
\right)^{\frac12},
\\
S_{\alpha,\psi}(f \cdot 1_{{\mathbb R}^n \setminus 3Q})(x')
&=
\left(
\iint_{\Gamma_\alpha(0)}
|\psi_t(f \cdot 1_{{\mathbb R}^n \setminus 3Q})(x'+y)|^2\frac{dydt}{t^{n+1}}
\right)^{\frac12}.
\end{align*}
We consider the weighted $L^2$-space
$L^2\left(\frac{dydt}{t^{n+1}},\Gamma_\alpha(0)\right)$.
By the triangle inequality
for this normed space, we have
\[
{\rm I}
\le
\left(
\iint_{\Gamma_\alpha(0)}
|\psi_t(f \cdot 1_{{\mathbb R}^n \setminus 3Q})(x+y)
-\psi_t(f \cdot 1_{{\mathbb R}^n \setminus 3Q})(x'+y)|^2\frac{dydt}{t^{n+1}}
\right)^{\frac12}.
\]
Keeping in mind that
\begin{equation}\label{eq:211114-2}
|x+y-z|+(2+\alpha)t \ge
|x-z|+(2+\alpha)t-|y|\ge
|x-z|+t
\end{equation}
for all $(y,t) \in \Gamma_\alpha(0)$ and $x,z \in {\mathbb R}^n$,
we estimate
\begin{align*}
\lefteqn{
|\psi_t(f \cdot 1_{{\mathbb R}^n \setminus 3Q})(x+y)
-\psi_t(f \cdot 1_{{\mathbb R}^n \setminus 3Q})(x'+y)|
}
\\
&\lesssim
\int_{{\mathbb R}^n \setminus 3Q}
\frac{1}{(t+|x+y-z|)^n}
\varphi\left(\frac{|x-x'|}{t+|x+y-z|}\right)
w\left(\frac{t}{t+|x+y-z|}\right)|f(z)|dz
\\
&\lesssim
\alpha^n
\int_{{\mathbb R}^n \setminus 3Q}
\frac{1}{(t+|x-z|)^n}
\varphi\left(\frac{(2+\alpha)|x-x'|}{t+|x-z|}\right)
w\left(\frac{(2+\alpha)t}{|x-z|}\right)|f(z)|dz\\
&\lesssim
\alpha^n
\int_{{\mathbb R}^n \setminus 3Q}
\frac{1}{(t+|x-z|)^n}
\varphi\left(\frac{\sqrt{n}(2+\alpha)\ell(Q)}{t+|x-z|}\right)
w\left(\frac{(2+\alpha)t}{|x-z|}\right)|f(z)|dz.
\end{align*}
If we integrate
this estimate against $(y,t) \in \Gamma_\alpha(0)$
and use Minkowski's inequality,
we obtain
\[
{\rm I}
\lesssim\alpha^n
\int_{{\mathbb R}^n \setminus 3Q}
\left[
\int_0^\infty
\left\{
\frac{1}{(t+|x-z|)^n}
\varphi\left(\frac{\sqrt{n}(2+\alpha)\ell(Q)}{t+|x-z|}\right)
w\left(\frac{(2+\alpha) t}{|x-z|}\right)\right\}^2\frac{dt}{t}
\right]^{\frac12}|f(z)|dz.
\]
By a change of variables
and Lemma \ref{auxiliary}
(b) 
we obtain
\begin{align*}
\lefteqn{
\int_0^\infty
\left\{
\frac{1}{(t+|x-z|)^n}
\varphi\left(\frac{\sqrt{n}(2+\alpha)\ell(Q)}{t+|x-z|}\right)
w\left(\frac{(2+\alpha)t}{|x-z|}\right)\right\}^2\frac{dt}{t}
}\\
&\le
\frac{1}{|x-z|^{2n}}
\varphi\left(\frac{\sqrt{n}(2+\alpha)\ell(Q)}{|x-z|}\right)^2
\int_0^\infty
\left\{
\frac{1}{(t+1)^{n}}
w\left((2+\alpha)t\right)
\right\}^2\frac{dt}{t}
\\
&{\sim} [w]_{\rm{Dini}}^2\log(2+\alpha)
\frac{1}{|x-z|^{2n}}
\varphi\left(\frac{\sqrt{n}(2+\alpha)\ell(Q)}{|x-z|}\right)^2.
\end{align*}
So, from Lemma \ref{auxiliary} (c),
we have
\begin{align*}
{\rm I}
&\lesssim \alpha^n\log^{\frac12}(2+\alpha)\,[w]_{\rm{Dini}}
\int_{{\mathbb R}^n \setminus 3Q}
\frac{1}{|x-z|^{n}}
\varphi\left(\frac{\sqrt{n}(2+\alpha)\ell(Q)}{|x-z|}\right)|f(z)|dz
\\ 
&\lesssim\alpha^n\log^{\frac12}(2+\alpha)[w]_{\rm{Dini}}
\sum_{k=1}^\infty
\frac{1}{(2^k \ell(Q))^{n}}
\varphi\left(\frac{\sqrt{n}(1+\alpha)}{2^k}\right)
\int_{2^{k+1}Q \setminus 2^k Q}|f(z)|dz
\\
&\lesssim
\alpha^n\log^{\frac32}(2+\alpha)[w]_{\rm{Dini}}[\varphi]_{\rm{Dini}}{\mathcal M}f(x),
\end{align*}
proving (\ref{eq:220330-3}).

As mentioned,
we thus deduce from (\ref{eq:211114-1}) and (\ref{eq:220330-3})
\[
S_{\alpha,\psi}(f \cdot 1_{{\mathbb R}^n \setminus 3Q})(x)
\lesssim
\alpha^n\log^{\frac32}(2+\alpha)[w]_{\rm{Dini}}[\varphi]_{\rm{Dini}}
{\mathcal M}f(x)+S_{\alpha,\psi}(f \cdot 1_{{\mathbb R}^n \setminus 3Q})(x').
\]
By the triangle inequality
\[
S_{\alpha,\psi}(f \cdot 1_{{\mathbb R}^n \setminus 3Q})(x)
\lesssim
\alpha^n\log^{\frac32}(2+\alpha)[w]_{\rm{Dini}}[\varphi]_{\rm{Dini}}
{\mathcal M}f(x)+S_{\alpha,\psi}f(x')+
S_{\alpha,\psi}(f \cdot 1_{3Q})(x').
\]
If we take the square root, then we obtain
\begin{align*}
\lefteqn{
\sqrt{S_{\alpha,\psi}(f \cdot 1_{{\mathbb R}^n \setminus 3Q})(x)}
}\\
&\lesssim
\sqrt{
\alpha^n\log^{\frac32}(2+\alpha)[w]_{\rm{Dini}}[\varphi]_{\rm{Dini}}{\mathcal M}f(x)}+
\sqrt{S_{\alpha,\psi}f(x')}+
\sqrt{S_{\alpha,\psi}(f \cdot 1_{3Q})(x')}.
\end{align*}
Next, we take the average over $Q$ against $x'$
and use the Hardy--Littlewood maximal operator
${\mathcal M}$ to have
\begin{align*}
\sqrt{S_{\alpha,\psi}(f \cdot 1_{{\mathbb R}^n \setminus 3Q})(x)}
&\lesssim
\sqrt{\alpha^n\log^{\frac32}(2+\alpha)[\varphi]_{\rm{Dini}}
[w]_{\rm{Dini}}{\mathcal M}f(x)}\\
&\quad+
{\mathcal M}\left[\sqrt{S_{\alpha,\psi}f}\right](x)+
\fint_Q\sqrt{S_{\alpha,\psi}(f \cdot 1_{3Q})(x')}dx'.
\end{align*}
By Kolmogorov's inequality
(see Lemma \ref{lem:Kol}), we have
\begin{align*}
\lefteqn{
\sqrt{S_{\alpha,\psi}(f \cdot 1_{{\mathbb R}^n \setminus 3Q})(x)}
\lesssim
\sqrt{\alpha^n\log^{\frac32}(2+\alpha) [w]_{\rm{Dini}}[\varphi]_{\rm{Dini}}
{\mathcal M}f(x)}}\\
&\quad+
{\mathcal M}\left[\sqrt{S_{\alpha,\psi}f}\right](x)+
\sqrt{\alpha^n [w]_{\rm{Dini}}[\varphi]_{\rm{Dini}}
\langle f \rangle_{1,3Q}}.
\end{align*}
If we use the Hardy--Littlewood maximal operator
${\mathcal M}$ once again, then we have
\begin{align*}
\sqrt{S_{\alpha,\psi}(f \cdot 1_{{\mathbb R}^n \setminus 3Q})(x)}
&
\lesssim
\sqrt{\alpha^n\log^{\frac32}(2+\alpha) [w]_{\rm{Dini}}[\varphi]_{\rm{Dini}}
{\mathcal M}f(x)}+
{\mathcal M}\left[\sqrt{S_{\alpha,\psi}f}\right](x).
\end{align*}
If we square the above inequality,
we obtain
\[
S_{\alpha,\psi}(f \cdot 1_{{\mathbb R}^n \setminus 3Q})(x)
\lesssim \alpha^n\log^{\frac32}(2+\alpha) 
[w]_{\rm{Dini}}[\varphi]_{\rm{Dini}}
{\mathcal M}f(x)+
{\mathcal M}_{\frac12} \circ S_{\alpha,\psi}f(x)
\]
Since $Q$ is also arbitrary,
it follows that
\[
{\mathcal N}_{S_{\alpha,\psi}}f(x)
\lesssim \alpha^n\log^{\frac32}(2+\alpha)
[w]_{\rm{Dini}}[\varphi]_{\rm{Dini}}
{\mathcal M}f(x)+
{\mathcal M}_{\frac12} \circ S_{\alpha,\psi}f(x).
\]
Since the operators
${\mathcal M}$ and
${\mathcal M}_{\frac12} \circ S_{\alpha,\psi}$
are weak-$(1,1)$ bounded with the norms bounded by constant times $1$ and 
$\alpha^n\log^{\frac32}(2+\alpha) ([w]_{\rm{Dini}}[\varphi]_{\rm{Dini}}+\|S_{1,\psi}\|_{L^2 \to L^2})$,
respectively,
it follows that
${\mathcal N}_{S_{\alpha,\psi}}$ enjoys the same
boundedness property as
${\mathcal M}_{\frac12} \circ S_{\alpha,\psi}$.
\end{proof}

\subsection{Proof of Lemma \ref{domiTheorem}}\label{sparse-lastpart} 
Now, we present the proof of Lemma \ref{domiTheorem}.

We define
\[
{\mathcal M}_{S_{\alpha,\psi},Q_0}f
:=
{\mathcal M}_{S_{\alpha,\psi}}(f \cdot 1_{{3Q_0}}), \quad
\widetilde{\mathcal M}_{S_{\alpha,\psi,Q_0}}f
:=
\max\{S_{\alpha,\psi} f ,{\mathcal M}_{S_{\alpha,\psi,Q_0}}f\}.
\]

By Theorem
\ref{prop:201226-2}
and
Lemma
\ref{weakestiamte},
 we have
\begin{equation}\label{eq:220425-111}
\|\widetilde{ \Mm}_{{S}_{\alpha, \psi,Q_0}}\|_{L^1 \to L^{1,\vc}} \le
\alpha^n\log^{\frac32}(2+\alpha)
(D_1+D_2)[w]_{\rm{Dini}}[\varphi]_{\rm{Dini}}.
\end{equation}
We now define
$$
E:=\Big\{ x\in Q_0: \widetilde{ \Mm}_{{S}_{\alpha, \psi,Q_0}}\Big(f\cdot 1_{3Q_0} \Big)(x) 
> \gamma^{\frac12}([w]_{\rm{Dini}}[\varphi]_{\rm{Dini}}+\|S_{1,\psi}\|_{L^2 \to L^2})
\langle f \rangle_{1,3Q_0}\Big\},
$$
where $\gamma$ is a constant which will be specified shortly. 

From this point, 
the rest of the proof is almost the same as \cite[p. 21]{BBD}.
Since the operator $\widetilde{ \Mm}_{{S}_{\alpha, \psi,Q_0}}$ is of weak type $(1,1)$
thanks to (\ref{eq:220425-111}), we have
$$
|E| \le \alpha^n\log^{\frac32}(2+\alpha)\gamma^{-\frac12}(D_1+D_2)|Q_0|.
$$
Thus,
if we choose $\gamma=2^{2n+2}(D_1+D_2)^2\alpha^{2n}\log^3(2+\alpha)$,
then
$|E| \le 2^{-n-1}|Q_0|$.
By a standard argument using the maximality
(see \cite{Duoandikoetxea01}, for example), 
we can find a family of dyadic cubes $\{ P_j\}_j \subset \mathscr{D}(Q_0)$ 
with the following properties:
\begin{enumerate}
 \item $\{ P_j\}_j$ is pairwise disjoint and $E\subset\bigcup\limits_j P_j$ almost everywhere,
that is,
\[
\left|E \setminus \bigcup_j P_j\right|=0.
\]
 \item
Let $\hat{P}_j $ be the parent of $P_j$.
Then each $P_j$
 satisfies
\[
\frac{|\hat{P}_j \cap E|}{|\hat{P}_j|}
\le
\frac{1}{2^{n+1}}<
\frac{|P_j \cap E|}{|P_j|}.
\]
 Hence in particular $|\hat{P}_j|\le 2^n |E|\le \frac12|Q_0|$ and
 $\hat{P}_j \subsetneq Q_0$.
 \end{enumerate} 
 We now decompose
 \begin{equation}\label{37-BBD}
 S_{\alpha, \psi}^2f\cdot 1_{Q_0}=S_{\alpha, \psi}^2f\cdot 1_{E}+S_{\alpha, \psi}^2f\cdot 1_{Q_0\setminus E}.
 \end{equation}
 Since $S_{\alpha, \psi}f\leq \widetilde{ \Mm}_{{S}_{\alpha, \psi,Q_0}}f$,
we have
 \begin{equation}\label{38-BBD}
 S_{\alpha, \psi}^2f \cdot 1_{Q_0\setminus E} 
 \leq \gamma ([\varphi]_{\rm{Dini}} [w]_{\rm{Dini}}+\|S_{1,\psi}\|_{L^2 \to L^2})^2
 \langle f \rangle_{1,3Q_0}{}^2 \cdot 1_{Q_0\setminus E}
 \end{equation}
in view of the definition of $E$.
 For the term $S_{\alpha, \psi}^2f \cdot 1_{E}$,
 using property (1), we have
\begin{equation}\label{39-BBD}
\begin{split}
 S_{\alpha, \psi}^2f \cdot 1_{ E} &\leq \sum_{j} \Big[ \Big|S_{\alpha, \psi}^2\Big(f\cdot 1_{3\hat{P}_j}\Big) \Big| \cdot 1_{ P_j}
+ \Big|S_{\alpha, \psi}^2f -S_{\alpha, \psi}^2\Big(f\cdot 1_{3\hat{P}_j}\Big) \Big| \cdot 1_{ P_j} \Big].
 \end{split}
 \end{equation}
We recall that $f$ is supported on $3Q_0$.
Thus,
\[
S_{\alpha, \psi}^2f -S_{\alpha, \psi}^2\Big(f\cdot 1_{3\hat{P}_j}\Big)
=S_{\alpha, \psi}^2\Big(f \cdot 1_{3Q_0} \Big) -S_{\alpha, \psi}^2
 \Big(f\cdot 1_{3\hat{P}_j}\Big).
\]

We write the quantity
${\mathcal M}_{S_{\alpha,\psi,Q_0}}f(x)$ out in full:
\[
{\mathcal M}_{S_{\alpha,\psi,Q_0}}f(x)
=
\sup_{Q}1_Q(x)
\sqrt{|S_{\alpha,\psi}(f \cdot 1_{3Q_0})(x)^2-S_{\alpha,\psi}(f \cdot 1_{3Q})(x)^2|}
\quad (x \in {\mathbb R}^n)
\]

 Since
$|\hat{P}_j \cap E| \le 2^{-n-1}|\hat{P}_j|$,
we have $\hat{P}_j \cap E^{\rm c} \neq \emptyset$.
Remark also that $3\hat{P}_j\subset 3Q_0$.
Thus, 
 we have
 \begin{align*}
 \Big|S_{\alpha, \psi}^2\Big(f \cdot 1_{3Q_0} \Big) -S_{\alpha, \psi}^2
 \Big(f\cdot 1_{3\hat{P}_j}\Big) \Big| 
 &\leq \inf_{w \in \hat{P}_j} \widetilde{ \Mm}^2_{{S}_{\alpha, \psi,Q_0}}
 \Big(f\cdot 1_{3Q_0} \Big)(w) 
 \\
 &\leq \gamma ([\varphi]_{\rm{Dini}}[w]_{\rm{Dini}}+\|S_{1,\psi}\|_{L^2 \to L^2})^2
 \langle f \rangle_{1,3Q_0}{}^2 
 \end{align*}
 thanks to property (2).
 This along with \eqref{39-BBD} and property (1), implies that
\begin{equation*}
\begin{split}
 S_{\alpha, \psi}^2f \cdot 1_{ E} &\leq \sum_{j} S^2_{\alpha, \psi}
 \Big(f\cdot 1_{3\hat{P}_j}\Big) \cdot 1_{ \hat{P}_j}
+ \gamma ([\varphi]_{\rm{Dini}}[w]_{\rm{Dini}}+\|S_{1,\psi}\|_{L^2 \to L^2})^2 \sum_{j}
\langle f \rangle_{1,3Q_0}{}^2 \cdot 1_{P_j}
\\
 &=\sum_{j} 
 S_{\alpha, \psi}^2\Big(f\cdot 1_{3\hat{P}_j}\Big) \cdot 1_{ \hat{P}_j}
+ \gamma ([\varphi]_{\rm{Dini}}[w]_{\rm{Dini}}+\|S_{1,\psi}\|_{L^2 \to L^2})^2
\langle f \rangle_{1,3Q_0}{}^2 \cdot 1_{E}.
 \end{split}
 \end{equation*}
 Combining this, \eqref{37-BBD} and \eqref{38-BBD}, we have
 $$
 S_{\alpha, \psi}^2f\cdot 1_{Q_0} 
 \leq \sum_{j} 
 S_{\alpha, \psi}^2\Big(f\cdot 1_{3\hat{P}_j}\Big) 
 \cdot 1_{ \hat{P}_j}
+ \gamma ([\varphi]_{\rm{Dini}}[w]_{\rm{Dini}}+\|S_{1,\psi}\|_{L^2 \to L^2})^2
\langle f \rangle_{1,3Q_0}{}^2 \cdot 1_{Q_0}.
 $$
 Iterating this estimate, we immediately get \eqref{Sparse-Dini} with 
 ${\mathcal S}:=\{P_j^k\}_{k \in \Z^{+},j}$, where $P_j^0= Q_0$, 
 $\{ P_j^1 \}_j = \{ \hat{P}_j \}_j$ and $\{ P_j^k \}_j$ are the cubes obtained 
 at the $k$-th stage of the iterative process. 
 The sparseness of ${\mathcal S}$ is straightforward from the choice of $\gamma$
 and property (3).
This completes our proof.

\section{The Multilinear Littlewood--Paley operators}\label{s7}
We define multi-linear Littlewood--Paley operators
in Section \ref{sect3.1}.
We formulate the main result
for multi-linear Littlewood--Paley operators in Section \ref{sect3.2}.
Section \ref{Application--weighted bounds}
presents corollaries of the theorem
in Section \ref{sect3.2}.
Section \ref{sect3.3} collects an auxiliary estimate.
Sections \ref{sect3.4}--\ref{s:Bad} are oriented to the proof of
one of the main theorems given in Section \ref{sect3.2},
where a weak type estimate is proved.
Section \ref{sect3.4} is a setup,
Section \ref{s:Good} deals with the good part of the functions.
Section \ref{s:Bad} concludes with the estimate of bad parts.
We assume $m=2$ in Sections \ref{sect3.4}--\ref{s:Bad}
for the sake of simplicity.
We give its sparse domination in Section \ref{s33}. 
We remark that
 the Dini-condition of $\varphi^{\frac1m}$ and $w^{\frac1m}$
is needed, but in the sparse bound
 the Dini-condition of $\varphi$ and $w$
suffices.
We need a stronger type of Dini condition,
that is, the one of $\varphi^{\frac1m}$ and $w^{\frac1m}$
when we consider the weak endpoint estimate
via the Calder\'{o}n--Zymgund decomposition.
See Lemma \ref{lem:220330-103},
where we consider the vector-valued functions
made up of more than $1$ bad parts.
\subsection{The definition of multilinear square functions} 
\label{sect3.1}

Let us recall the definition of multilinear square functions considered in this paper. 
Let $\psi(x,\vec{y}):=\psi(x, y_1, \ldots, y_m)$ be a locally integrable function defined away 
from the diagonal $x =y_1=\cdots =y_m$ in $(\mathbb{R}^n)^{m+1}$. 
Let $\varphi$ and $w$ be moduli of continuity. 

We assume that there is
a positive constant $A$ so that the following conditions hold:

\begin{itemize}
\item Size condition: 
\begin{equation}\label{eq:220330-205}
|\psi(x, \vec{y})| \leq A\left( 1+\sum_{i=1}^m |x-y_i|\right)^{-nm} w \left( \frac{1}{ 1+\sum_{i=1}^m |x-y_i|} \right).
\end{equation}
\item Smoothness condition: 
\begin{align*}
\lefteqn{
|\psi(x, \vec{y})-\psi(x', \vec{y})|
}\\
&\leq A\left( 1+\sum_{i=1}^m |x-y_i|\right)^{-nm} 
w \left( \frac{1}{ 1+\sum_{i=1}^m |x-y_i|} \right) \varphi \left( \frac{|x-x'|}{ 1+\sum_{i=1}^m |x-y_i|} \right),
\end{align*}
whenever $|x-x'|< \frac{1}{2}\max_j|x-y_j|$, and
\begin{align*}
\lefteqn{
\left|\psi(x, \vec{y})-\psi(x,y_1,\ldots,y'_i,\ldots,y_m)\right|
}\\
& \leq A\left( 1+\sum_{i=1}^m |x-y_i|\right)^{-nm} 
w \left( \frac{1}{ 1+\sum_{i=1}^m |x-y_i|} \right) \varphi \left( \frac{|y_i-y'_i|}{ 1+\sum_{i=1}^m |x-y_i|} \right),
\end{align*}
whenever $|y_i-y'_i|< \frac{1}{2}\max_j|x-y_j|$ for $i=1,2,\ldots,m$.
\end{itemize}
Let
\begin{equation}\label{eq:L}
L:=t+ |x+z-y_1|+|x+z-y_2|.
\end{equation}
For $t>0$, define the operator $\psi_t$ by
\[
\psi_t\vec{f}(x) :=\frac{1}{t^{nm}}\int_{(\Rn)^m}\psi
\Big(\frac{x}{t},\frac{y_1}{t},\ldots, \frac{y_m}{t}\Big)\prod_{j=1}^m f_j(y_j)dy_j,
\]
for all $\vec{f}=(f_1,\ldots,f_m)\in \S(\Rn) \times \cdots \times \S(\Rn)$.

Given $ \alpha>0$ and $\lambda>2m$, the multilinear square functions 
$S_{\alpha,\psi}$ and $g^*_{\lambda}$ are defined by
\begin{align}\label{eq:220330-101}
S_{\alpha,\psi}\vec{f}(x) 
:=\bigg(\iint_{\Gamma_\alpha(x)}|\psi_t\vec{f}(y)|^2\frac{dydt}{t^{n+1}}
\bigg)^{\frac12}
\end{align} and
\begin{align}\label{eq:220330-102}
g^*_{\lambda}\vec{f}(x) :=\bigg(\iint_{{\mathbb R}^{n+1}_+}\Big(\frac{t}{t+|x-y|}\Big)^{n\lambda}|\psi_t\vec{f}(y)|^2\frac{dydt}{t^{n+1}}\bigg)^{\frac12}.
\end{align}
Hereafter, we assume that there exist some exponents $1\le p_1,\ldots, p_m\le \infty$ and 
some $0<p<\infty$ with 
\begin{equation}\label{eq:220406-111}
\frac 1p=\frac{1}{p_1}+\cdots+\frac{1}{p_m}
\end{equation} 
such 
that $S_{1,\psi}$ maps continuously 
$L^{p_1}(\Rn)\times\cdots\times L^{p_m}(\Rn)$ to $L^p(\Rn)$. 
Under this condition, 
as a preparatory step,
we will establish that $S_{\alpha, \psi}$ and 
$g^*_{\lambda}$ maps continuously 
$L^1(\Rn) \times \cdots \times L^1(\Rn) \rightarrow L^{1/m,\infty}(\Rn)$ 
provided $\lambda > 2m$ and $\alpha>0$.

 \subsection{Main theorem}
\label{sect3.2}
 Based on the linear case, we consider
the multilinear case.
 We handle weak-$(1,1)$ estimates
of $S_{\alpha,\psi}$ and $g_{\lambda,\psi}^*$
given by
(\ref{eq:220330-101})
and
(\ref{eq:220330-102}),
respectively.

Also, by \cite[Lemma 3.1]{BuiHormozi} we have 
\begin{equation}\label{eq:220330-103}
\|S_{\alpha,\psi}\vec{f}\|_{L^{1/m,\infty}}
 \lesssim_{m}\alpha^{nm}\|S_{1,\psi}\vec{f}\|_{L^{1/m,\infty}}
 \end{equation}
 for all $\vec{f} \in L^1({\mathbb R}^n)^m$.
 \begin{theorem}[Weak end-point estimates]\label{Thm-multi-w}
Let $\lambda>2m.$
 Let $\sqrt[m]{\varphi},\,\sqrt[m]{w}$ be moduli of continuity satisfying
 the Dini condition. 
 Then for any $\rho>0$
and
$\alpha \ge 1$, 
 \[
\|S_{\alpha,\psi}\|_{L^{1} \times L^{1} \times \cdots \times L^{1} \to L^{1/m,\infty}}
\lesssim
\alpha^{nm} (A[\sqrt[m]{w}]^m_{\rm{Dini}}(1+ [\sqrt[m]{\varphi}]^m_{\rm{Dini}})
+\|S_{1,\psi}\|_{L^{p_1} \times L^{p_2} \times \cdots \times L^{p_m} \to L^p})
\]
and
\[
\|g^*_{\lambda, \psi}\|_{L^{1} \times L^{1} \times \cdots \times L^{1} \to L^{1/m,\infty}}
\lesssim A
[\sqrt[m]{w}]^m_{\rm{Dini}}(1+ [\sqrt[m]{\varphi}]^m_{\rm{Dini}})
+\|S_{1,\psi}\|_{L^{p_1} \times L^{p_2} \times \cdots \times L^{p_m} \to L^p}.
\]
 \end{theorem} 

Let $x \in {\mathbb R}^n$
We note that 
\begin{align}\label{grelationS}
 g_{\lambda, \psi}^{*}\vec{f}(x)
&\le
\bigg(\iint_{\Gamma(x,t)}\Big(\frac{t}{t+|x-y|}\Big)^{n\lambda}|\psi_t\vec{f}(y)|^2\frac{dydt}{t^{n+1}}\bigg)^{\frac12}
\nonumber\\
&\quad+
\bigg(\iint_{\Gamma(x,2^kt) \setminus \Gamma(x,2^{k-1}t)}\Big(\frac{t}{t+|x-y|}\Big)^{n\lambda}|\psi_t\vec{f}(y)|^2\frac{dydt}{t^{n+1}}\bigg)^{\frac12}
\nonumber\\
&\lesssim
\sum_{k=0}^\infty 2^{-k n\lambda}
\bigg(\iint_{\Gamma(x,2^k t)}|\psi_t\vec{f}(y)|^2\frac{dydt}{t^{n+1}}\bigg)^{\frac12}
\nonumber\\
& \lesssim \sum_{k=0}^{\infty} 2^{-\frac{k\lambda n} {2}} S_{ 2^{k+1}, \psi}\vec{f}(x)
\end{align} 
 for all $\vec{f}=(f_1, \ldots, f_m) \in L^1({\mathbb R}^n)^m$.
 Hence,
 in the light of 
 (\ref{eq:220330-103}), 
 it is enough to establish
\begin{equation}\label{eq:220426-2}
\rho^{\frac{1}{m}}
\left|\left\{x \in \Rn: S_{1,\psi}\vec{f}(x)>\rho\right\}\right| 
\lesssim A
[\sqrt[m]{w}]^m_{\rm{Dini}}(1+ [\sqrt[m]{\varphi}]^m_{\rm{Dini}})
+\|S_{1,\psi}\|_{L^{p_1} \times L^{p_2} \times \cdots \times L^{p_m} \to L^p}
\end{equation}
 for
 $\rho>0$ and 
 $\vec{f}=\{f_j\}_{j=1}^m \in L^1({\mathbb R}^n)$ with $\|f_j\|_{L^1}=1$. 
\subsection{Application--weighted bounds}
\label{Application--weighted bounds}

Before we come to the proofs of these theorems,
we present applications.
For a general account of multiple weights and related results, 
we refer the reader
to \cite{LOPTT}. We briefly introduce some definitions that we will need.

\medskip
 
 \par Consider $m$ weights $w_1,\dotsc,w_m$ and denote 
$\overrightarrow{w}=(w_1,\dotsc,w_m)$. Also let $1<p_1,\dotsc,p_m<\infty$ and 
$p$ be numbers such that $\frac{1}{p}=\frac{1}{p_1}+\dotsb+\frac{1}{p_m}$ and 
denote $\overrightarrow p = (p_1,\dotsc, p_m)$. Set

 $$\nu_{\vec w}:=\prod_{i=1}^m w_i^{\frac{p}{p_i}}.$$

We say that $\vec{w}$ satisfies the $A_{\vec{p}}$-condition if
\begin{equation}\label{eq:multiap_LOPTT}
 [\vec{w}]_{A_{\vec{p}}}:=\sup_{Q}\Big(\fint_Q\nu_{\vec w}\Big)
 \prod_{j=1}^m\Big(\fint_Q w_j^{1-p'_j}\Big)^{\frac{p}{p_j'}}<\infty.
\end{equation}
The class
$A_{\vec{p}}$
collects all $\vec{w}$
for which
$[\vec{w}]_{A_{\vec{p}}}$ is finite.
When $p_j=1$, $\Big(\fint_Q w_j^{1-p'_j}\Big)^{\frac{p}{p_j'}}$ is 
understood as $\displaystyle(\inf_Qw_j)^{-p}$.
 
 Assume that there exist some exponents
$1\leq p_1,\dotsc,p_m \leq \infty$
and some $0<p<\infty$
with $\frac{1}{p}=\frac{1}{p_1}+\dotsb+\frac{1}{p_m}$,
such that $S_{1, \psi}$ maps continuously
$L^{p_1}(\Rn)\times\dotsb\times L^{p_m}(\Rn)\to L^{p}(\Rn)$. 
 Similar to the techniques given in \cite{BuiHormozi, CY}, one can prove the 
 following theorems:
\begin{theorem}\label{thm:S-sharp}
Let $\sqrt[m]{\varphi},\,\sqrt[m]{w}$ be functions satisfying 
the Dini condition.
Let $\alpha\geq 1$,
 $\vec{w} \in A_{\vec{p}}$ and $\frac1p=\frac{1}{p_1}+\cdots +\frac{1}{p_m}$ 
with $1<p_1,\ldots,p_m<\infty$. 
Write
\begin{align*}
K
&:=\alpha^{nm}
\log^{\frac12+m}(2+\alpha)
(A[\sqrt[m]{w}]^m_{\rm{Dini}}(1+ [\sqrt[m]{\varphi}]^m_{\rm{Dini}})
+\|S_{1,\psi}\|_{L^{p_1} \times L^{p_2} \times \cdots \times L^{p_m} \to L^p})\\
&\quad \times[\vec{w}]_{A_{\vec{p}}}^{\max(\frac{1}{2},\frac{p_1'}{p},\ldots,
\frac{p_m'}{p})}.
\end{align*}
Then
 \begin{equation}\label{eq:BH}
\|S_{\alpha,\psi}\vec{f}\|_{L^{p}(\nu_{\vec{w}})} 
\lesssim K\prod_{i=1}^m\|f_i\|_{L^{p_i}(w_i)}
\end{equation}
for all $\vec{f}=\{f_i\}_{i=1}^m$
satisfying $f_i \in L^{p_i}(w_i)$ for each $i$,
where the implicit constant is independent of $\alpha$ and $\vec{w}$. 
\end{theorem}

We can mix Theorem \ref{domiCor1}
and Theorem \ref{thm:S-sharp}.
Moreover, 
we can prove the variant of $g_{\lambda,\psi}^*$.
\begin{theorem}
 \label{thm:sharp2}
 Let $w$ and $\varphi$ be functions satisfying the Dini condition.
Let $\lambda> 2m$, $\vec{p}=(p_1,\ldots,p_m)$ with $1<p_1,\dotsc,p_m<\infty$ 
and $1/p_1+\dotsb+1/p_m=1/p$. If $\vec{w}=(w_1,\dotsc,w_m)\in A_{\vec{p}}$, 
then
 \begin{equation}
 \label{eq:sharp2}
\|g^*_{\lambda, \psi}\vec{f}\|_{L^p(\nu_{\vec{w}})} 
 \lesssim K \prod_{i=1}^m \|f_i\|_{L^{p_i}(w_i)}
 \end{equation}
for all $\vec{f}=\{f_i\}_{i=1}^m$ 
satisfying $f_i \in L^{p_i}(w_i)$ for each $i$,
where the implicit constant is independent of $\vec{w}$
and $K$ is as above. 
\end{theorem}

\subsection{Key lemmas}
\label{sect3.3}


Here and below we write
$d\vec{y}=dy_1 \cdots dy_m$.
We use the counterpart of Lemma \ref{auxiliary} to the multilinear setting:
 \begin{lemma}\label{Yabuta} Let $w$ be a modulus of continuity 
satisfying
the Dini condition. Then for $\alpha \geq 1$ and $t>0$, we have
\begin{eqnarray*}
&&
\sup_{|z-x|<\alpha t} \int_{\Rnm} \left(t+\sum_{i=1}^{m}|z-y_i| \right)^{-nm}w \left(\frac{t}{t+\sum_{i=1}^{m}|z-y_i| } \right) \Big| \prod_{i=1}^m f_i(y_i) \Big| d\vec{y}
\\ 
&&\lesssim
\alpha^{n m}\log(2+\alpha)
[w]_{\rm Dini}
\prod_{i=1}^m{\mathcal M}f_i(x)
\end{eqnarray*}
 for all $x \in {\mathbb R}^n$
 and
for all $\vec{f}=\{f_i\}_{i=1}^m \subset L^1(\Rn)$.
 \end{lemma}
 \begin{proof} 
 Fix $x$.
 Let $(z,t) \in \Gamma_\alpha(x)$.
We calculate
\begin{align*}
2m\alpha\left(t+\sum_{i=1}^m|z-y_i|\right)
&\ge
m t+m|z-x|+\sum_{i=1}^m|z-y_i|
\ge
t+\sum_{i=1}^m|x-y_i|.
\end{align*}
We estimate
\begin{align*}
\lefteqn{\int_{\Rnm} 
 \left(t+\sum_{i=1}^{m}|z-y_i| \right)^{-nm}w \left(\frac{t}{t+\sum_{i=1}^{m}|z-y_i| } \right) 
 \Big| \prod_{i=1}^m f_i(y_i) \Big| d\vec{y}
 }\\
&\lesssim
\frac{w(1)}{t^{nm}}
\int_{Q(x,t)^m}
\prod_{i=1}^m|f_i(y_i)|d\vec{y}\\
&\quad+
\sum_{k=1}^\infty
\frac{\alpha^{nm}}{2^{k n m}t^{nm}}
\int_{Q(x,2^k t)^m \setminus Q(x,2^{k-1}t)^m}
w(2^{2-k}m\alpha)
\prod_{i=1}^m|f_i(y_i)|d\vec{y}\\
&\lesssim
w(1)
\prod_{i=1}^m{\mathcal M}f_i(x)
+
\sum_{k=1}^\infty
\alpha^{nm}w(2^{2-k}m\alpha)
\prod_{i=1}^m{\mathcal M}f_i(x)\\
&\lesssim
\alpha^{n m}\log(2+\alpha)
[w]_{\rm Dini}
\prod_{i=1}^m{\mathcal M}f_i(x),
\end{align*}
as required.
 \end{proof}

The next step is the core of the proof
of the weak-$(1,1)$ boundedness.
For functions $\theta_1$ and $\theta_2$,
write
$$
\widetilde{S}_{1,\psi}(\theta_1,\theta_2)(x) :=\iint_{\Gamma_1(0) } \Big| \frac{1}{t^{2n}} \int_{\Rt} \psi \Big(
\frac{x+z}{t}, \frac{y_1}{t}, \frac{y_2}{t} \Big) \theta_{1}(y_1)\theta_2(y_2) d\vec{y} \Big| \frac{dz dt}{t^{n+1}} .
$$
We set
\[I_{\theta_1,\theta_2}(x,z,t) := \Big| \frac{1}{t^{2n}} \int_{\Rt} \psi \Big(
\frac{x+z}{t}, \frac{y_1}{t}, \frac{y_2}{t} \Big) \theta_{1}(y_1)\theta_2(y_2) d\vec{y} \Big|.
\]
Let
\begin{align*}
A^{(k)}(x)&:=
\sqrt{
\frac{[{w}]_{\rm Dini}}{|x-c(Q_{k})|^{2n}}
\varphi\left(\frac{2\sqrt{n}\ell(Q_{k})}{|x-c(Q_{k})|}\right)}\\
B^{(k)}(x)&:=
\sqrt{
\int_{\ell(Q_{k})}^\infty
\frac{|Q_k|}{(t+|x-c(Q_{k})|)^{2n}}
w\left(\frac{t}{t+|x-c(Q_{k})|}\right)\frac{dt}{t^{n+1}}
}.
\end{align*}

We estimate
$I_{\theta_1,\theta_2}(x,z,t)$
as follows:
\begin{lemma}\label{lem:220406-1111}Let $\theta_{1}$ and $\theta_{2}$ be integrable functions.
Assume
that
$\theta_{1}$ and $\theta_{2}$ are supported on cubes
$Q_1$ and $Q_2$, respectively and that
\[
\int_{{\mathbb R}^n}\theta_{1}(x)dx=
\int_{{\mathbb R}^n}\theta_{2}(x)dx=0.
\]
Then, for $x\in \mathbb R^n\setminus (64nQ_1\cup 64nQ_2)$
\begin{align}
\iint_{\Gamma_1(0)} I_{\theta_1,\theta_2}(x,z,t) \frac{dz dt}{t^{n+1}}
\label{eq:220402-111}
\lesssim
\prod_{k=1}^2
\left(
A^{(k)}(x)
+
B^{(k)}(x)
\right)
\|\theta_{k}\|_{L^1}.
\end{align}
\end{lemma}

\begin{proof}
Let $x\in \mathbb R^n\setminus (64nQ_1\cup 64nQ_2)$. 
We write
$$
E_{k}(x):= \Big\{z \in \Rn: |x+z-c(Q_{k})|<16n \ell(Q_{k}) \Big\}
$$
for each $k=1,2$.

Let
\begin{equation}\label{eq:220506-1}
L:=t+ |x+z-y_1|+|x+z-y_2|.
\end{equation}
We
partition
$\Gamma_1(0)$ into $4$ domains:
\begin{align*}
\Gamma_1(0)
&=
\Gamma_1(0) \cap ((E_{1}(x) \cap E_{2}(x)) \times {\mathbb R})\\
&\quad
\cup
\Gamma_1(0) \cap ((E_{1}(x)^{\rm c} \cap E_{2}(x)) \times {\mathbb R})\\
&\quad
\cup
\Gamma_1(0) \cap ((E_{1}(x) \cap E_{2}(x)^{\rm c}) \times {\mathbb R})\\
&\quad
\cup
\Gamma_1(0) \cap ((E_{1}(x)^{\rm c} \cap E_{2}(x)^{\rm c}) \times {\mathbb R}).
\end{align*}
\begin{itemize}
\item
{\bf (a)} Let $z \in E_{1}(x) \cap E_{2}(x)$.
Then
similar to 
(\ref{eq:220402-1}),
we have
\begin{equation}\label{eq:220402-4}
t>16\max(\ell(Q_{1}),\ell(Q_{2})).
\end{equation}
Recall that $L$ is given by 
(\ref{eq:L}).
It 
follows 
from the size condition
and
(\ref{eq:220330-207}) that
\begin{align*}
I_{\theta_1,\theta_2}(x,z,t) 
&\leq \int_{\Rt}\frac{1}{L^{2n}}w \left( \frac{t}{L} \right) 
\Big|\theta_1(y_1) \theta_2(y_2) \Big| d\vec{y}.
\end{align*}
Let $k=1,2$.
We estimate
\[
L
\ge
\frac12t+\frac12|x-y_k|
\ge
\frac1{4}(t+|x-c(Q_{k})|)
\]
using (\ref{eq:220402-4}).
The result is:
\[
I_{\theta_1,\theta_2}(x,z,t) 
\lesssim
\frac{1}{(t+|x-c(Q_{k})|)^{2n}}w \left( \frac{4t}{t+|x-c(Q_{k})|} \right)
\times
\|\theta_1\|_{L^1}\|\theta_2\|_{L^2}.
\]
We integrate this inequality over
$\Gamma_1(0)\cap (E_1(x)\times \mathbb R)$.
\begin{equation}\label{eq:220402-1112}
\iint_{\Gamma_1(0)\cap (E_1(x)\times \mathbb R)} I_{\theta_1,\theta_2}(x,z,t) 
\frac{dz dt}{t^{n+1}}
\lesssim
\|\theta_1\|_{L^1} \|\theta_2\|_{L^1}B^{(k)}(x)^2.
\end{equation}
If we take the geometric mean of 
estimate
(\ref{eq:220402-1112})
over $k=1,2$,
then we have a term
that is included
in the right-hand side of
(\ref{eq:220402-111}).
\item
{\bf (b)} Let $z \in E_{2}(x) \setminus E_{1}(x)$.
In this case, we have
\[
\frac{|y_1-c(Q_{1})|}{t}\le\frac{n\ell(Q_{1})}{t}
<\frac{|x+z-c(Q_{1})|}{8t}.
\]
From the smoothness condition
and the moment condition of $\theta_1$,
we obtain
\begin{align*}
\lefteqn{
I_{\theta_1,\theta_2}(x,z,t)
}\\
&=\Big| \frac{1}{t^{2n}} \int_{\Rt} \Big( \psi \Big(
\frac{x+z}{t}, \frac{y_1}{t}, \frac{y_2}{t}\Big) - \psi \Big(
\frac{x+z}{t}, \frac{c(Q_{1})}{t}, \frac{y_2}{t} 
\Big)
\Big)
 \theta_1(y_1)\theta_2(y_2) d\vec{y} \Big|\\
&\lesssim \int_{\Rt} \frac{1}{L^{2n}}w \left( \frac{t}{L} \right)
\varphi \left( \frac{n\ell(Q_{1})}{L} \right)
 \Big|\theta_1(y_1) \theta_2(y_2) \Big| d\vec{y}\\
&\lesssim 
\|\theta_1\|_{L^1} \|\theta_2\|_{L^1}\times
\frac{1}{(t+|x-c(Q_{1})|)^{2n}}
\varphi\left(\frac{4n\ell(Q_{1})}{|x-c(Q_{1})|}\right)
w\left(\frac{4t}{t+|x-c(Q_{1})|}\right)
\end{align*}
for all $t>0$
from 
(\ref{eq:220330-201}).
If we integrate this estimate against
$(z,t) \in \Gamma_0$,
then we obtain
\begin{equation}\label{eq:220426-1}
\iint_{\Gamma_1(0)}
I_{\theta_1,\theta_2}(x,z,t)\frac{dzdt}{t^{n+1}}
\lesssim
\|\theta_1\|_{L^1} \|\theta_2\|_{L^1}
A^{(1)}(x)^2.
\end{equation}
Meanwhile,
(\ref{eq:220402-1112}) is still available 
for $k=2$.
Thus,
if we take the geometric mean of 
estimates
(\ref{eq:220402-1112})
with $k=2$
and
(\ref{eq:220426-1}),
then we have 
\[
\iint_{\Gamma_1(0)}
I_{\theta_1,\theta_2}(x,z,t)\frac{dzdt}{t^{n+1}}
\lesssim
\|\theta_1\|_{L^1} \|\theta_2\|_{L^1}
A^{(1)}(x)B^{(2)}(x),
\]
which is included
in the right-hand side of
(\ref{eq:220402-111}).

\item
{\bf (c)} Let $z \in E_{1}(x) \setminus E_{2}(x)$.
Simply swap the role of $j_1$ and $j_2$
in {\bf (b)} to have
\[
\iint_{\Gamma_1(0)}
I_{\theta_1,\theta_2}(x,z,t)\frac{dzdt}{t^{n+1}}
\lesssim
\|\theta_1\|_{L^1} \|\theta_2\|_{L^1}
B^{(1)}(x)A^{(2)}(x).
\]
\item
{\bf (d)} Let $z \in E_{1}(x)^{\rm c} \cap E_{2}(x)^{\rm c}$.
Argue as in {\bf (b)} to have
\[
\iint_{\Gamma_1(0)}
I_{\theta_1,\theta_2}(x,z,t)\frac{dzdt}{t^{n+1}}
\lesssim
\|\theta_1\|_{L^1} \|\theta_2\|_{L^1}
A^{(k)}(x)^2
\]
for $k=1,2$.
As before, we take the geometric mean
over
$k=1,2$ to have a term
that is included
in the right-hand side of
(\ref{eq:220402-111}).
\end{itemize}
\end{proof}

\subsection{Calder\'on--Zygmund decomposition--Setup}
\label{sect3.4}

Here and below for the sake of simplicity, we assume $m=2$.
 Suppose $\vec{f}=(f_1, f_2) \in L^1({\mathbb R}^n)^2$
with 
\begin{equation}\label{eq:220330-206}
\|f_i\|_{L^1}=1 \quad (i=1,2)
\end{equation}
and $\rho > 0$ is fixed. 
We will 
form the Calder\'on--Zygmund decomposition of each $f_i$
at height $\rho^{\frac12}$
as in \cite{GT1}. For each $i=1,2$, we write 
$f_i=g_i + b_i$, where there is a 
disjoint
collection of dyadic cubes 
$\{ Q_{i,j} \}_j \subset {\mathcal D}({\mathbb R}^n)$ 
such that $\supp b_{i,j} \in Q_{i,j} $ and each $b_i= \sum_{j} b_{i,j}$. Moreover, we have 
\begin{equation}\label{eq:220407-11111}
\int_{Q_{i,j}} b_{i,j}(x) dx =0, 
\end{equation}
\begin{equation}\label{eq:220407-11112}
\Big| \bigcup_j Q_{i,j} \Big| \lesssim \rho^{-\frac12},\end{equation}
\begin{equation}\label{eq:220407-11113}
 \|b_{i,j}\|_{L^1} \lesssim \|f_i\|_{L^1(Q_{i,j})},\end{equation}
\begin{equation}\label{eq:220407-11114}
\|g_i\|_{L^1}\le1,\end{equation}
 and
\begin{equation}\label{eq:220407-11115}
\|g_i\|_{L^\infty} \lesssim \rho^{\frac12}.\end{equation}
In particular, by interpolation
between
(\ref{eq:220407-11114})
and
(\ref{eq:220407-11115}),
we have
\begin{equation}\label{eq:220330-204}
 \|g_i\|_{L^{p_i}} \lesssim \rho^{\frac{1}{2 p'_i}}.
\end{equation}
Let
$$
\Omega= \bigcup_{k=1}^2 \left( \bigcup_{j} 64n Q_{k,j} \right).
$$
We summarize what we need for the proof of the theorem:
\begin{lemma}\label{lem:Omega}
Let $k=1,2$ and $x \in \Rn \setminus \Omega$.
\begin{enumerate}
\item[$(1)$]
We have
\begin{equation}\label{eq:220330-209}
\Big| \Omega\Big| \lesssim \rho^{-\frac12}.
\end{equation}
\item[$(2)$]
For all $j$, we have
\begin{equation}\label{eq:220330-207}
|x-c(Q_{k,j})|> 32n \ell(Q_{k,j}).
\end{equation}
\end{enumerate}
\end{lemma}
\begin{proof}\
\begin{enumerate}
\item[$(1)$]
Thanks to 
(\ref{eq:220330-206})
and
(\ref{eq:220407-11112}),
\begin{align*}
\Big| \Omega\Big| \leq \sum_{k=1}^2 \sum_{j} |64n Q_{k,j}|
=(64n)^n \sum_{k=1}^2 \sum_{j} | Q_{k,j}| 
\lesssim 
\sum_{k=1}^2 \rho^{-\frac12} \lesssim \rho^{-\frac12}.
\end{align*}
\item[$(2)$]
This is clear from the definition of $\Omega$.
\end{enumerate}
\end{proof}

\begin{lemma}\label{lem:220330-111}
$
\left|\left\{x \in \Rn: {\mathcal M}b_1(x) {\mathcal M}b_2(x) >\rho\right\}\right|
\lesssim \rho^{-\frac12}.
$
\end{lemma}

\begin{proof}
By the weak-$(1,1)$ inequality for ${\mathcal M}$
and the normalization
(\ref{eq:220330-206}),
we have
\begin{align*}
\lefteqn{
\left|\left\{x \in \Rn: {\mathcal M}b_1(x) {\mathcal M}b_2(x) >\rho\right\}\right|
}\\ 
&\leq \left|\left\{x \in \Rn: {\mathcal M}b_1(x) >\rho^{\frac12} \right\}\right|
+\left|\left\{x \in \Rn: {\mathcal M}b_2(x) >\rho^{\frac12} \right\}\right|\\
&\lesssim \rho^{-\frac12}\|b_1\|_{L^1}+ \rho^{-\frac12}\|b_2\|_{L^1} \lesssim \rho^{-\frac12},
\end{align*}
as required.
\end{proof}

\subsection{Good part}
\label{s:Good}
Since $S_{1,\psi}$ is assumed bounded from
$L^{p_1}(\Rn)\times L^{p_2}(\Rn)$ to $L^p(\Rn)$,
we deduce from
(\ref{eq:220330-204})
\begin{align*}
\rho \left|\left\{x \in {\mathbb R}^n\,:\, S_{1,\psi}\vec{g}(x)>\rho\right\}\right|^{\frac1p} 
&\leq \|S_{1,\psi}\vec{g}\|_{L^{p}} \\
&\lesssim
\|S_{1,\psi}\|_{L^{p_1} \times L^{p_2} \to L^p} 
\prod_{i=1}^2 \|g_i\|_{L^{p_i}}\\
&\lesssim
\|S_{1,\psi}\|_{L^{p_1} \times L^{p_2} \to L^p} 
\prod_{i=1}^2 \rho^{\frac{1}{2p'_i}}\\
&=
\|S_{1,\psi}\|_{L^{p_1} \times L^{p_2} \to L^p} \rho^{1-\frac{1}{2p}},
\end{align*}
where $\vec{g}=(g_1,g_2)$. 
So,
we have
$$
\left|\left\{x\in {\mathbb R}^n\,:\, S_{1,\psi}\vec{g}(x)>\rho\right\}\right|\lesssim \rho^{-\frac12}.
$$
\subsection{Bad part}
\label{s:Bad}

To deal with other parts, 
we treat the case $m=2$ for the sake of simplicity. 
Thus, we have
$g_1,g_2,b_1,b_2$.

Therefore, 
in view of (\ref{eq:220330-209}),
the proof of Theorem \ref{Thm-multi-w} hinges
on the following three estimates
(\ref{YI1})--(\ref{YI3}) below:
\begin{lemma}\label{lem:220330-101}\
\begin{itemize}
\item
The bad-good estimate
\begin{equation}\label{YI1}
\left|\left\{x \in \Rn \setminus \Omega: S_{1,\psi}(b_1,g_2)(x)>\rho\right\}\right| 
\lesssim \rho^{-\frac12} [{w}]_{\rm{Dini}} [{\varphi}]_{\rm{Dini}}
\end{equation}
holds.
\item
The good-bad estimate
\begin{equation}\label{YI2}
\left|\left\{x \in \Rn \setminus \Omega: S_{1,\psi}(g_1,b_2)(x)>\rho\right\}\right| 
\lesssim \rho^{-\frac12} [{w}]_{\rm{Dini}} [{\varphi}]_{\rm{Dini}} 
\end{equation}
holds.
\end{itemize}
\end{lemma}

\begin{proof}
We concentrate on ($\ref{YI1}$)
due to similarity; simply swap the role of $b_1$ and $g_2$ for the proof of ($\ref{YI2}$).
Fix $j_1$.
 We estimate $S_{1,\psi}(b_{1,j_1},g_2)(x)$ for $\dist \left(x,c(Q_{1,j_1}) \right)> 64n\ell(Q_{1,j_1})$. We have
\begin{align*}
\lefteqn{
\Big| \frac{1}{t^{2n}} \int_{\Rt} \psi \Big(
\frac{x+z}{t}, \frac{y_1}{t}, \frac{y_2}{t} \Big) b_{1,j_1}(y_1)g_2(y_2) d\vec{y} \Big|
}\\
&=\Big| \frac{1}{t^{2n}} \int_{\Rt} \left(
 \psi \Big(\frac{x+z}{t}, \frac{y_1}{t}, \frac{y_2}{t} \Big)-\psi \Big(\frac{x+z}{t}, \frac{c(Q_{1,j_1})}{t}, \frac{y_2}{t} \Big)
\right) b_{1,j_1}(y_1)g_2(y_2) d\vec{y} \Big|.
\end{align*}
As in (\ref{eq:220402-1}), 
we have $t> 32n \ell(Q_{1,j_1})> \ell(Q_{1,j_1})$. Moreover, from $y_1 \in Q_{1,j_1} $, 
$|x-c(Q_{1,j_1})|> 64n\ell(Q_{1,j_1})$ and 
$(z,t) \in \Gamma_1(0)$, it follows that 
$$
2\left( t+ |x+z-y_1|\right)> t+ |x-y_1|+t-|z|> t+ |x-y_1|,
$$
and
\begin{align*}
2\left( t+ |x-y_1|\right)
&\geq 2t+ |x-c(Q_{1,j_1})|-|c(Q_{1,j_1})-y_1|\\ 
&\geq t+ |x-c(Q_{1,j_1})| +t- \frac{\sqrt{n} }{2} \ell(Q_{1,j_1})>t+ |x-c(Q_{1,j_1})|.
\end{align*}

We estimate 
\begin{align*}
\lefteqn{
\Big| \frac{1}{t^{2n}} \int_{\Rt} \psi \Big(
\frac{x+z}{t}, \frac{y_1}{t}, \frac{y_2}{t} \Big) b_{1,j_1}(y_1)g_2(y_2) d\vec{y} \Big|
}\\
&\lesssim \|g_2\|_{L^\infty} \int_{\Rt} \frac{1}{L^{2n}}w\left( \frac{t}{L} \right) |b_{1,j_1}(y_1)| d\vec{y} 
\\
 &\lesssim \rho^{\frac12} 
 \frac{1}{\left(t+|x-c(Q_{1,j_1})|\right)^{n}}w\left( \frac{4t}{t+|x-c(Q_{1,j_1})|} \right)\\
 &\quad \times
 \int_{\Rn} |b_{1,j_1}(y_1)| \left( \int_{\Rn} \Big(1+ \frac{|x+z-y_2|}{t+|x+z-y_1|} \Big)^{-2n} dy_2 \right) dy_1
\\
 &\lesssim \rho^{\frac12}\|b_{1,j_1}\|_{L^1}
 \frac{1}{\left(t+|x-c(Q_{1,j_1})|\right)^{n}}w\left( \frac{4t}{t+|x-c(Q_{1,j_1})|} \right).
\end{align*}

Hence, as in the linear case, 
\begin{align*}
\left(\iint_{\Gamma_1(0), |x+z-c(Q_{1,j_1})|<32n \ell(Q_{1,j_1}) } \Big| 
\int_{\Rt} \psi \Big(
\frac{x+z}{t}, \frac{y_1}{t}, \frac{y_2}{t} \Big) b_{1,j_1}(y_1)g_2(y_2) d\vec{y} \Big|^2 \frac{dz dt}{t^{3n+1}} \right)^{\frac12}\\
\lesssim \rho^{\frac12} \|b_{1,j_1}\|_{L^1} 
\sum_{k=1}^{\infty} \frac{1}{2^{\frac{kn}{2}}\Big( 2^k \ell(Q_{1,j_1})+ |x-c(Q_{1,j_1})| \Big)^n} 
w \left( \frac{2^{k+2} \ell(Q_{1,j_1})}{2^k \ell(Q_{1,j_1})+ |x-c(Q_{1,j_1})|} \right).
\end{align*}

Let $(z,t) \in \Gamma_1(0)$
satisfy $|x+z-c(Q_{1,j_1})|\geq 32n \ell(Q_{1,j_1})$. 
Also let $y_1 \in Q_{1,j_1} $ and assume $|x-c(Q_{1,j_1})|> 64n\ell(Q_{1,j_1})$. Then as before
$$
4\left( t+ |x+z-y_1|\right)> 2\left( t+ |x-y_1|\right) >t+|x-c(Q_{1,j_1})|.
$$
Moreover, 
\begin{align*}
|x+z-y_1|> |x+z-c(Q_{1,j_1})|-|y_1-c(Q_{1,j_1})| 
&> 32n \ell(Q_{1,j_1})-\frac{\sqrt{n} }{2} \ell(Q_{1,j_1}) \\
&\ge n \ell(Q_{1,j_1})> 2 \sqrt{n}|y_1-c(Q_{1,j_1})|.
\end{align*}
Therefore,
abbreviating
$t+|x+z-y_1|+|x+z-y_2|$
to
$L$
and 
using the above estimates and the smoothness condition for $\psi$, 
we have
\begin{align*}
\lefteqn{
\Big| \frac{1}{t^{2n}} \int_{\Rt} \psi \Big(
\frac{x+z}{t}, \frac{y_1}{t}, \frac{y_2}{t} \Big) b_{1,j_1}(y_1)g_2(y_2) d\vec{y} \Big|
}\\
&=\Big| \frac{1}{t^{2n}} \int_{\Rt} \left(
 \psi \Big(\frac{x+z}{t}, \frac{y_1}{t}, \frac{y_2}{t} \Big)-\psi \Big(\frac{x+z}{t}, \frac{c(Q_{1,j_1})}{t}, \frac{y_2}{t} \Big)
\right) b_{1,j_1}(y_1)g_2(y_2) d\vec{y} \Big|\\
&\lesssim \|g_2\|_{L^\infty}
\int_{\Rt} \frac{1}{L^{2n}}w\left( \frac{t}{L} \right)\varphi \left( \frac{|y_1-c(Q_{1,j_1})|}{L} 
 \right) |b_{1,j_1}(y_1)| d\vec{y} \\
 &\lesssim \|g_2\|_{L^\infty} w \left( \frac{4t}{t+|x-c(Q_{1,j_1})|} \right) \int_{\Rt}\frac{1}{L^{2n}} \varphi \left( \frac{2 \sqrt{n} \ell(Q_{1,j_1}) }{|x-c(Q_{1,j_1})|} \right) |b_{1,j_1}(y_1)| d\vec{y}.
 \end{align*}
 By a change of variables we obtain
 \begin{align*}
 \lefteqn{\Big| \frac{1}{t^{2n}} \int_{\Rt} \psi \Big(
\frac{x+z}{t}, \frac{y_1}{t}, \frac{y_2}{t} \Big) b_{1,j_1}(y_1)g_2(y_2) d\vec{y} \Big|
}\\
 &\lesssim \rho^{\frac12} \int_{\Rn} \frac{dy_2}{ \left( 1+|y_2| \right)^{2n} } \int_{\Rn} \frac{w \left( \frac{4t}{t+|x-c(Q_{1,j_1})|} \right) }{ ( 
t+|x+z-y_1| )^{n}} \varphi \left( \frac{2 \sqrt{n} \ell(Q_{1,j_1}) }{|x-c(Q_{1,j_1})|} \right)
 |b_{1,j_1}(y_1)| dy_1 \\
 &\sim \rho^{\frac12} \frac{1}{ \left( t+|x-c(Q_{1,j_1})| \right)^n} w \left( \frac{4t}{t+|x-c(Q_{1,j_1})|} \right) \varphi \left( \frac{2 \sqrt{n} \ell(Q_{1,j_1}) }{|x-c(Q_{1,j_1})|} \right) \|b_{1,j_1}\|_{L^1}.
 \end{align*}
 Hence, as in the linear case,
 \begin{eqnarray*}
&&\left(\int_{z \in \Gamma_1(0), |x+z-c(Q_{1,j_1})|\geq 32n \ell(Q_{1,j_1}) } \Big| \frac{1}{t^{2n}} \int_{\Rt} \psi \Big(
\frac{x+z}{t}, \frac{y_1}{t}, \frac{y_2}{t} \Big) b_{1,j_1}(y_1)g_2(y_2) d\vec{y} \Big|^2 \frac{dz dt}{t^{n+1}} \right)^{\frac12}\\
&&\lesssim \rho^{\frac{1}{2}}[{w}]_{\rm{Dini}} \frac{1}{|x-c(Q_{1,j_1})|^n} 
\varphi \left( \frac{2 \sqrt{n} \ell(Q_{1,j_1}) }{|x-c(Q_{1,j_1})|} \right) 
\|b_{1,j_1}\|_{L^1}. 
\end{eqnarray*} 
 Thus as in the linear case, we obtain
 $$
 \int_{\Rn \setminus \Omega} S_{1,\psi}(b_{1,j_1},g_2)(x) dx 
\lesssim \rho^{\frac12} [{w}]_{\rm{Dini}} [{\varphi}]_{\rm{Dini}} 
\|b_{1,j_1}\|_{L^1}
 $$
 and hence adding this estimate over $j$
 $$
 \int_{\Rn \setminus \Omega} S_{1,\psi}(b_1,g_2)(x) dx \lesssim \rho^{\frac12} 
[{w}]_{\rm{Dini}} [{\varphi}]_{\rm{Dini}}.
 $$
 From this and Chebychev's inequality, it follows 
 \begin{align*}
 \left|\left\{x \in \Rn \setminus \Omega: S_{1,\psi}(b_1,g_2)(x)>\rho\right\}\right| 
&\leq \frac{1}{\rho} \int_{\Rn \setminus \Omega} S_{1,\psi}(b_1,g_2)(x) dx 
\lesssim \frac{1}{\rho^{\frac{1}{2}}} 
[{w}]_{\rm{Dini}} [{\varphi}]_{\rm{Dini}}. 
 \end{align*} 
 
 This proves ($\ref{YI1}$). 
\end{proof}
\begin{lemma}\label{lem:220330-103}
The bad-bad estimate
\begin{equation}\label{YI3}
 \left|\left\{x \in \Rn \setminus \Omega: S_{1,\psi}(b_1,b_2)(x)>\rho
 \right\}\right| 
 \lesssim \rho^{-\frac12} 
([\sqrt{w}]_{\rm{Dini}}
+[\sqrt{w}]^{1/2}_{\rm{Dini}}[\sqrt{\varphi}]_{\rm{Dini}})
\end{equation}
holds.
\end{lemma}

For the proof, we employ the following notation:
Let $x \in \Rn \setminus \Omega$
and $(z,t) \in \Gamma_1(0)$. 
Also, let $k=1,2$.
\begin{itemize}
\item
We write 
$$
I_{j_1,j_2}(x,z,t)=\Big| \frac{1}{t^{2n}} \int_{\Rt} \psi \Big(
\frac{x+z}{t}, \frac{y_1}{t}, \frac{y_2}{t} \Big) b_{1,j_1}(y_1)b_{2,j_2}(y_2) 
d\vec{y} \Big|.
$$
\item
Let
\begin{align*}
A^{(k)}_j(x)&:=
\sqrt{
\frac{[\sqrt{w}]_{\rm Dini}}{|x-c(Q_{k,j})|^{2n}}
\varphi\left(\frac{4n\ell(Q_{k,j})}{|x-c(Q_{k,j})|}\right)}\\
B^{(k)}_{j}(x)&:=
\sqrt{
\int_{\ell(Q_{k,j})}^\infty
\frac{|Q_{k,j}|}{(t+|x-c(Q_{k,j})|)^{2n}}
w\left(\frac{t}{t+|x-c(Q_{k,j})|}\right)\frac{dt}{t^{n+1}}
}
\end{align*}

for each $j$.
\item
Also, define
\[
C^{(k)}(x):=\sum_{j} \|b_{k,j}\|_{L^1}(A^{(k)}_j(x)+B^{(k)}_j(x)).
\]
Note that
\[
\|C^{(k)}\|_{L^1(\mathbb R^n\setminus \Omega)} 
\lesssim ({[\sqrt{w}]_{\rm Dini}}
+
{[\sqrt{w}]}^{1/2}_{\rm Dini}[\sqrt{\varphi}]_{\rm Dini})
\sum_{j} \|b_{k,j}\|_{L^1}.
\]
\item
Let
$L:=t+ |x+z-y_1|+|x+z-y_2|$.
See (\ref{eq:220506-1}).
\end{itemize}
\begin{proof}
Let $x \in {\mathbb R}^n \setminus \Omega$.
As for $S_{1,\psi}(b_1,b_2)$, we apply Lemma \ref{Yabuta}. 
By
(\ref{eq:220330-205}) and Lemma \ref{Yabuta},
we have 
 \begin{align*} 
 \lefteqn{
S_{1,\psi}(b_1,b_2)(x)
}\\
&\leq 
\sup_{(z,t) \in \Gamma_1(0)} 
\sqrt{\frac{1}{t^{2n}} \int_{\Rt} \Big| \psi \Big(
\frac{x+z}{t}, \frac{y_1}{t}, \frac{y_2}{t} \Big) b_{1}(y_1)b_2(y_2) \Big| 
d\vec{y}} \sqrt{\widetilde{S}_{1,\psi}(b_1,b_2)(x)} 
\\
&\lesssim \sqrt{[{w}]_{\rm{Dini}} }
\sqrt{{\mathcal M}b_1(x) {\mathcal M}b_2(x)} 
\sqrt{\widetilde{S}_{1,\psi}(b_1,b_2)(x)}.
 \end{align*}

By the triangle inequality, we have
\begin{align*}
\widetilde{S}_{1,\psi}(b_1,b_2)(x)
\le
\sum_{j_1} \sum_{j_2} \widetilde{S}_{1,\psi}(b_{1,j_1}, b_{2,j_2})(x).
\end{align*}
Summing 
Lemma \ref{lem:220406-1111}
over $j_1$ and $j_2$, we get
\begin{align*}
\widetilde{S}_{1,\psi}(b_1,b_2)(x)\lesssim
\sum_{j_1,j_2}
\|b_{1,j_1}\|_{L^1} \|b_{2,j_2}\|_{L^1}
\prod_{k=1}^2
(A_{j_k}^{(k)}(x)+B_{j_k}^{(k)}(x))=C^{(1)}(x)C^{(2)}(x)
\end{align*}
for all $x \in \Omega$.

Then by the Cauchy--Schwarz inequality,
\begin{align*}
\int_{ \Rn \setminus \Omega}\widetilde{S}_{1,\psi}(b_1,b_2)(x)^{\frac12} dx 
&\lesssim
\prod_{k=1}^2\left(\int_{ \Rn \setminus \Omega} C^{(k)}(x) dx \right)^{\frac12} 
\\
&\lesssim ([\sqrt{w}]_{\rm{Dini}}+[\sqrt{w}]^{\frac12}_{\rm{Dini}}[\sqrt{\varphi}]_{\rm{Dini}})
\prod_{k=1}^2
\|b_k\|^{\frac12}_{L^1(\Rn \setminus \Omega)}
\\
&\lesssim ([\sqrt{w}]_{\rm{Dini}}
+[\sqrt{w}]^{\frac12}_{\rm{Dini}}[\sqrt{\varphi}]_{\rm{Dini}}).
\end{align*}
Therefore,
by the Chebychev inequality, 
\begin{equation}\label{eq:220402-1113}
\left|\left\{x \in \Rn \setminus \Omega: \widetilde{S}_{1,\psi}(b_1,b_2)(x)>\rho\right\}\right| 
\lesssim \rho^{-\frac12}
([\sqrt{w}]_{\rm{Dini}}
+[\sqrt{w}]^{\frac12}_{\rm{Dini}}[\sqrt{\varphi}]_{\rm{Dini}}),
\end{equation}
which means that (\ref{eq:220402-1113}),
or equivalently, ($\ref{YI3}$) holds. 
Here we have used the fact $[{w}]_{\rm{Dini}}\le[\sqrt{w}]^{2}_{\rm{Dini}}$. 

Now, set 
\begin{equation*}
\tilde{w}=\frac{w}{[\sqrt{w}]^2_{\rm Dini}},\quad 
\tilde \varphi = \frac{\varphi}{[\sqrt{\varphi}]^2_{\rm Dini}},\ \text{ and }\ 
\tilde A=A[\sqrt{w}]^2_{\rm Dini}(1+[\sqrt{\varphi}]^2_{\rm Dini}).
\end{equation*}
Then the size condition and the smoothness condition for $\psi$ 
in the definition of $S_{\alpha,\psi}$ are satisfied with $\tilde A$, 
$\tilde{w}$ and $\tilde{\varphi}$ in place of $A$, ${w}$ and ${\varphi}$, 
respectively. 
In this case, $[\sqrt{\tilde{w}}]_{\rm Dini}\le 1$ and 
$[\sqrt{\tilde\varphi}]_{\rm Dini}< 1$. 
Thus, by Lemmas \ref{lem:Omega}, \ref{lem:220330-111}, \ref{lem:220330-101}, 
and \ref{lem:220330-103}, we see that 
\begin{equation*}
\rho^{\frac{1}{2}}
\left|\left\{x \in \Rn : \widetilde{S}_{1,\psi}(f_1,f_2)(x)>\rho\right\}\right| 
\lesssim
(A[\sqrt{w}]_{\rm Dini}(1+[\sqrt{\varphi}]_{\rm Dini})
+1+\|S_{1,\psi}\|_{L^{p_1}\times L^{p_2}\to L^p})\|f_1\|_{L^1}\|f_2\|_{L^1}. 
\end{equation*}
A homogeneity argument yields
\begin{equation*}
\rho^{\frac{1}{2}}
\left|\left\{x \in \Rn : \widetilde{S}_{1,\psi}(f_1,f_2)(x)>\rho\right\}\right| 
\lesssim
(A[\sqrt{w}]_{\rm Dini}(1+[\sqrt{\varphi}]_{\rm Dini})
+\|S_{1,\psi}\|_{L^{p_1}\times L^{p_2}\to L^p})\|f_1\|_{L^1}\|f_2\|_{L^1}.
\end{equation*}
The proof of Theorem \ref{Thm-multi-w} in the case $m=2$
is therefore complete. 
\end{proof}

\subsection{Sparse domination for multilinear Littlewood--Paley operators}\label{s33}

We pass to the multilinear case.
Once again we work in the setting of the $m$-linear case
$m \ge 2$.
\begin{theorem}\label{domiTheorem2}Let
$\alpha \ge 1$,
$Q_0 \in {\mathcal D}$.
Suppose that we have
integrable functions
$f_1,f_2,\ldots,f_n$
supported in $3Q_0$.
Then there exists a sparse family $\S \subset {\mathcal D}(Q_0)$ 
$($depending on $\vec{f}$$)$ 
such that 
\begin{align}\label{Sparse-main}
S_{\alpha,\psi} \vec{f} \cdot 1_{Q_0}
&\lesssim \alpha^{nm}
\log^{\frac12+m}(2+\alpha) 
[{w}]_{\rm Dini}(1+[{\varphi}]_{\rm Dini}) 
\bigg[\sum\limits_{P \in \S} 
\prod_{i=1}^{m} \langle f_i \rangle_{1,P}^2 1_P\bigg]^{\frac12}.
\end{align}
\end{theorem}
\medskip

\begin{proof} The proof is almost identical to the linear case, so we only 
sketch the proof.
Let $x \in \Rn$. Consider
\[
{\mathcal M}_{S_{\alpha,\psi}}\vec{f}(x)
:=
\sup_{Q}1_Q(x)
\sqrt{|S^2_{\alpha,\psi}\vec{f}(x)-S_{\alpha,\psi}^2(\vec{f} \cdot 1_{3Q})(x)|}.
\]
and
\[
{\mathcal N}_{S_{\alpha,\psi}}\vec{f}(x)
:=
\sup_{Q}1_Q(x)
S_{\alpha,\psi}(\vec{f} \cdot 1_{{\mathbb R}^n \setminus 3Q})(x),
\]
where $Q$ moves over all cubes. 
Fix a cube $Q$ so that $x \in Q$.
Let $x' \in Q$ be arbitrary. Keeping in mind that
\[
|x+z-y_i|+(2+\alpha)t \ge
|x-y_i|+(2+\alpha)t-|z|\ge
|x-y_i|+t
\]
for all $(z,t) \in \Gamma_\alpha(0)$ and $y_i \in {\mathbb R}^n$. We get
\begin{align*}
{\rm I}&:=
|
S_{\alpha,\psi}(\vec{f} \cdot 1_{{\mathbb R}^n \setminus 3Q})(x)
-
S_{\alpha,\psi}(\vec{f} \cdot 1_{{\mathbb R}^n \setminus 3Q})(x')
|\\
 &\le
\left(
\iint_{\Gamma_\alpha(0)}
|\psi_t(\vec{f} \cdot 1_{{\mathbb R}^n \setminus 3Q})(x+z)
-\psi_t(\vec{f} \cdot 1_{{\mathbb R}^n \setminus 3Q})(x'+z)|^2
\frac{dzdt}{t^{n+1}}
\right)^{\frac12}\\
&\lesssim\Big[ \int_{(\Rn \setminus 3Q)^m} \int_0^\infty\\
& \ \ \ \times 
\left\{
\frac{\alpha^{nm}\prod_{i=1}^{m}|f_i(y_i)|}{(t+\sum_{i=1}^{m}|x-y_i|)^{nm}}
\varphi\left(\frac{\sqrt{n}(2+\alpha)\ell(Q)}{t+\sum_{i=1}^{m}|x-y_i|}\right)
w\left(\frac{(2+\alpha) t}{t+\sum_{i=1}^{m}|x-y_i|}\right)d\vec{y}\right\}^2 
\frac{dt}{t}\Big]^{\frac12}
\\
& {\leq} \alpha^{nm} \log^{\frac12} (2+\alpha) [{w}]_{\rm{Dini}} 
\int_{(\Rn \setminus 3Q)^m} 
\frac{1}{(\sum_{i=1}^{m}|x-y_i|)^{nm}}
\varphi\left(\frac{\sqrt{n}(2+\alpha)\ell(Q)}{\sum_{i=1}^{m}|x-y_i|}\right)
 \prod_{i=1}^{m}|f_i(y_i)|d\vec{y}.
\end{align*}
Here we have used Lemma \ref{auxiliary} (b) in the last step.
By Lemma \ref{auxiliary} (c),
\begin{align*}
{\rm I} &\lesssim \alpha^{nm} \log^{\frac12} (2+\alpha) [{w}]_{\rm{Dini}} 
\sum_{k=1}^{\infty}
{\varphi}\left(\frac{\sqrt{n}(2+\alpha)}{2^k}\right)
 \prod_{i=1}^{m} {\mathcal M}f_i(x)
 \\
 & {\lesssim} \alpha^{nm} \log^{\frac12 + m} (2+\alpha) [{w}]_{\rm{Dini}} 
 [{\varphi}]_{\rm{Dini}} \prod_{i=1}^{m} {\mathcal M}f_i(x).
\end{align*}
As a result, 
\begin{align*}
\lefteqn{
S_{\alpha,\psi}(\vec{f} \cdot 1_{{\mathbb R}^n \setminus 3Q})(x)
}\\
&\lesssim \alpha^{nm} \log^{\frac12 + m} (2+\alpha) 
[{w}]_{\rm{Dini}} [{\varphi}]_{\rm{Dini}} 
\prod_{i=1}^{m} {\mathcal M}f_i(x)
+
S_{\alpha,\psi} \vec{f}(x')+
S_{\alpha,\psi}(\vec{f} \cdot 1_{3Q})(x').
\end{align*}

By Kolmogorov's inequality
(see Lemma \ref{lem:Kol})
$$
\int_{E} S_{\alpha,\psi}\vec{f}(x)^{\frac{1}{2m}}dx \lesssim
|E|^{\frac12} 
\Big( 
\alpha^{nm}\|S_{1,\psi}\|_{L^1 \times L^1 \times \cdots \times 
L^1 \to L^{1,\infty}}
 \prod_{i=1}^{m} \|f_i\|_{L^1} \Big)^{\frac{1}{2m}}.
$$
We write
$$
J:=S_{\alpha,\psi}(\vec{f} \cdot 
1_{{\mathbb R}^n \setminus 3Q})(x)^{\frac{1}{2m}}.
$$
Taking the average over $Q$ against $x'$
and using the Hardy--Littlewood maximal operator
${\mathcal M}$ twice, we obtain 
\begin{align*}
J 
&\lesssim \alpha^{\frac{n}{2}} \log^{ \frac12+\frac1{4m} } (2+\alpha) 
\left( [{w}]_{\rm{Dini}} [{\varphi}]_{\rm{Dini}} 
\prod_{i=1}^{m} {\mathcal M}f_i(x) \right)^{\frac{1}{2m}} \\
&\quad+
\alpha^{\frac{n}{2}}
\|S_{1,\psi}\|_{L^1 \times L^1 \times \cdots \times L^1 \to L^{1,\infty}}
{\mathcal M}_{\frac{1}{2m}} \circ S_{\alpha,\psi}\vec{f}(x).
\end{align*}

Since $Q$ is also arbitrary,
it follows that
\begin{align*}
{\mathcal N}_{S_{\alpha,\psi}}\vec{f}(x)
&\lesssim \alpha^{nm} \log^{ \frac12+m} (2+\alpha) 
\left( [{w}]_{\rm{Dini}} [{\varphi}]_{\rm{Dini}} 
\prod_{i=1}^{m} {\mathcal M}f_i(x) \right)^{\frac{1}{2m}} \\
&\quad+
\alpha^{nm}
\|S_{1,\psi}\|_{L^1 \times L^1 \times \cdots \times L^1 \to L^{1,\infty}}
{\mathcal M}_{\frac{1}{2m}} \circ S_{\alpha,\psi}\vec{f}(x).
\end{align*}
Since
${\mathcal M}_{\frac{1}{2m}}$ is bounded
on 
$L^{\frac1m,\infty}({\mathbb R}^n)$, ${\mathcal M}_{\frac{1}{2m}} 
\circ S_{\alpha,\psi}\vec{f}$ is bounded 
from $L^1(\Rn) \times \cdots \times L^1(\Rn) $ to $L^{\frac1m, \infty}(\Rn)$. 
So, there exists a constant $D>0$,
independent of $\alpha\ge1$, such that
\begin{eqnarray*}
&&\|{\mathcal N}_{S_{\alpha,\psi}}\|_{L^1 \times \cdots \times L^1 \to L^{\frac1m, \infty}}+
\|{\mathcal M}_{S_{\alpha,\psi}}\|_{L^1 \times \cdots \times L^1 \to L^{\frac1m, \infty}}\\
&&\leq D(\alpha^{nm}\log^{\frac12+m}(2+\alpha)
[{w}]_{\rm Dini}[{\varphi}]_{\rm Dini}+
\alpha^{nm}
\|S_{1,\psi}\|_{L^1 \times L^1 \times \cdots \times L^1 \to L^{1,\infty}}).
\end{eqnarray*}
As the last step of the proof, one can 
follow the idea in Section \ref{sparse-lastpart} very closely and use the above estimate. 
We omit further details.
\end{proof}

\section{Examples}
\label{s10.11}
Let $m=1$.
In this section, we 
exhibit some examples of $\psi$ and moduli of continuity $w$ and $\varphi$ 
for which
(\ref{eq-sizecondition}),
(\ref{eq-smoothcondition1})
and
(\ref{eq-smoothcondition2})
hold and $S_{\alpha,\psi}$ is $L^2$-bounded. 

For the sake of simplicity, we consider the operator $S_{1,\psi}$. 
Indeed, 
Lemma \ref{l2boundedness}
allows us to handle different apertures $\alpha>0$.

\subsection{Example 1.}
\label{s10.112}
We discuss how different the Dini condition and the log-Dini condition are.

Here we list a function
$\psi$ for which
(\ref{eq-sizecondition}),
(\ref{eq-smoothcondition1})
and
(\ref{eq-smoothcondition2})
hold.

In fact, letting $\kappa>1$,
we define
\[
\psi(x)=\psi(x_1,x_2,\ldots,x_n)
=
\frac{\sin x_1}{(1+|x|^2)^{\frac{n}{2}}\log^{\kappa}(2+|x|^2)}
\quad (x \in {\mathbb R}^n).
\]
We show that 
${\mathcal F}\psi$ decays rapidly at $\infty$ and $0$ for any $n$ if $\kappa>1$.
\begin{lemma}\label{lem:220404-1}
Let $l \in {\mathbb N}$.
Then
\[
|{\mathcal F}\psi(\xi)|
\lesssim \frac{1}{1+|\xi|^l}\log^{1-\kappa}\left(2+\frac{1}{|\xi|}\right)
\quad (\xi \in {\mathbb R}^n).
\]
\end{lemma}

\begin{proof}
Since $\kappa>1$,
$\psi,\nabla^l \psi \in L^1(\Rn)$
for any $l \in {\mathbb N}$.
Thus, denoting by ${\mathcal F}$ the Fourier transform,
we have
\begin{equation}\label{eq:220404-1}
|{\mathcal F}\psi(\xi)| \lesssim \frac{1}{1+|\xi|^l}
\quad (\xi \in {\mathbb R}^n)
\end{equation}
for any $l \in {\mathbb N}$.

Let $|\xi|<1$.
We will seek a finer estimate than 
(\ref{eq:220404-1}) by paying attention to the expression
\[
{\mathcal F}\psi(\xi)
=
\frac{1}{(2\pi)^{\frac{n}{2}}}
\int_{{\mathbb R}^n}
\frac{\sin x_1}{(1+|x|^2)^{\frac{n}{2}}\log^{\kappa}(2+|x|^2)}
e^{-i x \cdot \xi}dx.
\]
Since $\psi$ is an odd function, we have
\[
|{\mathcal F}\psi(\xi)|
\lesssim
\int_{{\mathbb R}^n}
\frac{\min(1,|x|\cdot|\xi|)}{(1+|x|^2)^{\frac{n}{2}}\log^{\kappa}(2+|x|^2)}dx
\sim
\int_0^\infty \frac{t^{n-1}\min(1,t|\xi|)}{(1+t^2)^{\frac{n}{2}}\log^\kappa(2+t^2)}dt.
\]
We estimate this integral similar to Lemma \ref{auxiliary}.
We make a change of variables $t=e^u$
to have
\begin{align*}
|{\mathcal F}\psi(\xi)|
&\lesssim
\int_{-\infty}^\infty \frac{e^{n u}\min(1,e^u|\xi|)}
{(1+e^{2u})^{\frac{n}{2}}(\log(2+e^{2u}))^\kappa}du
\\
&= 
\int_{-\infty}^{\log\frac{1}{|\xi|}} 
\frac{e^{(n+1) u}|\xi|}{(1+e^{2u})^{\frac{n}{2}}\log^\kappa(2+e^{2u})}du
+
\int_{\log\frac{1}{|\xi|}}^\infty 
\frac{e^{n u}}{(1+e^{2u})^{\frac{n}{2}}\log^\kappa(2+e^{2u})}du.
\end{align*}
We decompose the integral into $4$ parts:
\begin{align*}
\lefteqn{
|{\mathcal F}\psi(\xi)|
}\\
&\lesssim 
|\xi|\biggl(\int_{-\infty}^{0} 
\frac{e^{(n+1) u}\,du}{(1+e^{u})^{n}}
+\int_{0}^{\log\frac{1}{\sqrt{|\xi|}}}\frac{e^u\,du}{(1+u)^\kappa}
+\int_{\log\frac{1}{\sqrt{|\xi|}}}^{\log\frac{1}{|\xi|}}
\frac{e^u\,du}{(1+u)^\kappa}\biggr)
+\int_{\log\frac{1}{|\xi|}}^\infty \frac{du}{(1+u)^\kappa}.
\end{align*}
Since
\[
e^u \le \frac{1}{\sqrt{|\xi|}} \quad \left(0 \le u \le \log \frac{1}{\sqrt{|\xi|}}\right)
\]
and
\[
\frac{1}{(1+u)^\kappa} 
\le \left(1+\log\frac{1}{\sqrt{|\xi|}}\right)^{-\kappa} 
\quad \left(\log \frac{1}{\sqrt{|\xi|}} \le u \le \log \frac{1}{{|\xi|}}\right),
\]
we obtain
\begin{align*}
|{\mathcal F}\psi(\xi)|
&\lesssim 
|\xi|\int_{0}^{1} \frac{t^n}
{1+t^n}dt+\sqrt{|\xi|}
+\left(1+\log\frac{1}{\sqrt{|\xi|}}\right)^{-\kappa} 
+\int_{\log\frac{1}{|\xi|}}^\infty \frac{du}{(1+u)^\kappa}
\\
&\lesssim \log^{1-\kappa}\left(2+\frac{1}{|\xi|}\right).
\end{align*}
Thus, the proof is complete.
\end{proof}
Let $\kappa>\frac32$.
With Lemma \ref{lem:220404-1} in mind, we consider the integral operator
\[
S_{1, \psi} f (x)=\Big(\iint_{\Gamma_1(x)}|f\star \psi_t(y)|^2
\f{dydt}{t^{n+1}}\Big)^{\frac12},
\]
which was defined in (\ref{eq:220330-2})
with $\phi$ replaced by $\psi$.

We can estimate the $L^2$-norm of $S_{1,\psi}$ with ease
by the use of the Fourier transform:
\[
\|S_{1,\psi}f\|_{L^2}
\lesssim
\|f\|_{L^2}
\sqrt{
\int_0^\infty
\min\left(\frac{1}{t^{2}},
\log^{2-2\kappa}\left(2+\frac{1}{t}\right)
\right)\frac{dt}{t}
}
\sim
\|f\|_{L^2}.
\]

We also deal with an estimate for $\log$.
\begin{lemma} \label{memo-star}
Let $h,x \in \Rn$.
Assume $|h| \leq \frac{|x|}{2}$. 
Then
for any $0<\gamma \leq 1$, we have
$$
\frac{\min \Big(1, |h|^\gamma \Big)}{\log (2+|x|)}\log \Big(2+\frac{1+|x|}{|h|} \Big) \lesssim_\gamma 1.
$$
That is, there exists a constant $C_\gamma$ which depends only on $\gamma$ such that
$$
\frac{\min \Big(1, |h|^\gamma \Big)}{\log (2+|x|)}\log \Big(2+\frac{1+|x|}{|h|} \Big) \le C_\gamma.
$$
\end{lemma}
\begin{proof}We distinguish two cases.\\
\textbf{Case 1.} Let $|h| \leq \frac{1}{2}$. In this case $\log \frac{1}{|h|} \geq \log 2$ and 
\begin{align*}
\log \Big(2+\frac{1+|x|}{|h|} \Big) &\leq \log \Big(\frac{2+|h|+|x|}{|h|} \Big) \leq \frac{2}{\log 2} \log \Big( 2+|h|+|x| \Big) \log \frac{1}{|h|}\\
&\leq \frac{2}{\log 2} \log \Big( 2+|h|+|x| \Big)\frac{1}{\gamma}\cdot
\frac{1}{|h|^\gamma} \lesssim_\gamma \frac{\log (2+|x|)}{|h|^\gamma}.
\end{align*}
\textbf{Case 2.} Let $|h| \geq \frac{1}{2}$. Then
$$
\log \Big(2+\frac{1+|x|}{|h|} \Big) \leq 2 \log (2+ |x|). 
$$
So,
we obtain the desired result. 
\end{proof}

We now suppose that $\kappa>2$.
Let us verify that
(\ref{eq-sizecondition}),
(\ref{eq-smoothcondition1})
and
(\ref{eq-smoothcondition2})
hold
for
\begin{equation}\label{eq:220404-111}
w(t)=\varphi(t):=\log^{-\frac{\kappa}{2}}\left(2+\frac{1}{\min(1,t)}\right).
\end{equation}
It is noteworthy that $w$ fails the $\log$-Dini condition
if $2\le \kappa\le 4$.
Condition (\ref{eq-sizecondition})
is easy to check since $|\sin x_1| \le 1$ for all $x_1 \in {\mathbb R}$
and $\kappa>2$.

Since
\[
|\nabla \psi(x+h)| \lesssim 
\frac{1}{(1+|x|^2)^{\frac{n}{2}}\log^{\kappa}(2+|x|^2)}
\]
and
\[
\frac{|h|}{\log(2+|x|)} \lesssim \frac{1}{\log\left(2+\frac{1+|x|}{|h|}\right)}
\]
if $|h|<\frac{|x|}{2}$
(see Lemma \ref{memo-star}),
(\ref{eq-smoothcondition1})
and
(\ref{eq-smoothcondition2})
are satisfied.

\subsection{Example 2.}
The next example illustrates that
the product type Dini condition is useful.
Let $\kappa>1$ and $\beta \in {\mathbb R}$ satisfy $\beta-\kappa<0$.
We set
\[
\psi(x)=\psi(x_1,x_2,\ldots,x_n)
:=
\frac{\sin x_1}{(1+|x|^2)^{\frac{n}{2}}\log^{\kappa}(2+|x|^2)}
\quad (x \in {\mathbb R}^n),
\]
and
 \[
w(t):=\log^{\beta-\kappa}\Big( 1+\frac1{\min(1,t)} \Big).
\quad \varphi(t):=\log^{-\beta} \Big(1+\frac1{\min(1,t)} \Big)
\]
and before.

Let $\kappa>2$, $\beta>1$
satisfy $\kappa-\beta>1$.
 The functions $\psi$, $w$ and $\varphi$ enjoy the following properties:
 \begin{itemize}
 \item[{(a)}] Both $w$ and $\varphi$ are moduli of continuity satisfying
 the Dini condition.
 \item[{(b)}] We have
 $\|S_{1,\psi}f\|_{L^2} \lesssim
\|f\|_{L^2}$
for all $f \in L^2({\mathbb R}^n)$
 \item[{(c)}] We have
 \[ |\psi(x)| \lesssim {\mathcal M}1_{Q(0,1)}(x) w\Big( \frac{1}{ 1+|x|}\Big)
 \quad (x \in {\mathbb R}^n).
 \]
 Thus,
 (\ref{eq-sizecondition})
is satisfied.
 \item[{(d)}] For $x, h \in \Rn$ with $|h| \leq \frac{|x|}{2}$, we have 
 \[ |\psi(x+h)-\psi(x)| \lesssim {\mathcal M}1_{Q(0,1)}(x) w\Big( \frac{1}{ 1+|x|}\Big)\varphi \Big( \frac{|h|}{ 1+|x|}\Big).
 \]
 Thus,
(\ref{eq-smoothcondition1})
and
(\ref{eq-smoothcondition2})
are satisfied.
 \end{itemize}

\begin{proof} The proof of (a) and (c) is straightforward. Item (b) is proved in Example 1. We now focus on item (d). Set
\[
g(x)=g(x_1,x_2,\ldots,x_n)
=
\frac{1}{(1+|x|^2)^{\frac{n}{2}}\log^{\kappa}(2+|x|^2)}
\quad (x \in {\mathbb R}^n).
\]
Then 
\begin{align*}
|\psi(x+h)-\psi(x)| &= \Big( \sin(x_1+h_1)-\sin(x_1)\Big) g(x)+\sin(x_1+h_1)\Big(g(x+h)-g(x) \Big)\\
&=:I_1(x,h)+I_2(x,h).
\end{align*}
By Lemma \ref{memo-star}
with $\gamma=\frac{1}{\beta}$, we have
\begin{align*}
|I_1(x,h)| &\lesssim \min (1,|h|) {\mathcal M}1_{Q(0,1)}(x)\log^{-\kappa}(2+|x|)\\
&\lesssim 
{\mathcal M}1_{Q(0,1)}(x) 
\log^{\beta-\kappa} (2+|x| )
\log^{-\beta} \Big( 2+\frac{1+|x|}{h} \Big) \\
&= {\mathcal M}1_{Q(0,1)}(x)w\Big( \frac{1}{ 1+|x|}\Big)\varphi \Big( \frac{|h|}{ 1+|x|}\Big).
\end{align*}

We will estimate $I_2(x,h)$.
We write
\begin{align*}
J_1(x,h)&:= \Big(\log(2+|x-h|^2) \Big)^{-\kappa}\Big|
\Big(1+|x+h|^2\Big)^{-\frac{n}{2}} 
-
\Big(1+|x|^2\Big)^{-\frac{n}{2}}
 \Big|\\
J_2(x,h)&:=\Big(1+|x|^2\Big)^{-\frac{n}{2}} 
 \Big| \Big(\log(2+|x+h|^2) \Big)^{-\kappa}
 - \Big(\log(2+|x|^2) \Big)^{-\kappa} 
 \Big|.
\end{align*}
Then
\begin{align*}
|I_2(x,h)|
&\leq |g(x+h)-g(x)|
\le
J_1(x,h)+J_2(x,h).
\end{align*}

Recall that $2|h| \le |x|$.
So, we have
\begin{align*}
J_1(x,h)
&\lesssim \min (1,|h|) {\mathcal M}1_{Q(0,1)}(x)\log^{-\kappa}(2+|x|)\\
J_2(x,h) &\lesssim\min (1,|h|) {\mathcal M}1_{Q(0,1)}(x)\log^{-\kappa}(2+|x|).
\end{align*}
by the mean value theorem.
Thus, in a similar way to the calculation of $I_1(x,h)$, 
the same bound can be obtained for $J_1(x,h)$ and $J_2(x,h)$. 
The proof of item (d)
is therefore complete. 
\end{proof}

\subsection{Example 3.} 
Let
$w$ and $\varphi$ be given by
(\ref{eq:220404-111}).
Also let $\kappa>2$.
A similar observation as in Example 1 shows that
\[
\psi(x)
:=
\frac{\partial}{\partial x_1}
(1+|x|^2)^{-\frac{n-1}{2}}\log^{-\kappa}(2+|x|^2)\quad (x \in {\mathbb R}^n)
\]
satisfies
(\ref{eq-sizecondition}),
(\ref{eq-smoothcondition1})
and
(\ref{eq-smoothcondition2}).

\section{Concluding remarks}
\label{sect5}
 
Recall that the Marcinkiewicz function was a key tool
in the paper \cite{SXY}. It seems
useful to refine the boundedness property
in the spirit of this paper.
In connection with the Dini condition,
the following general estimate seems of use for further work:
 \begin{remark}
 Let $w$ be a modulus of continuity. 
Let
$$F_{w}(x)=\sum_{k=1}^\infty \lambda_k{\mathcal M}1_{Q(c_k,r_k)}(x) w\left(\frac{r_k}{r_k+|x-c_k|} \right)
\quad (x \in {\mathbb R}^n) $$
be 
the generalized Marcinkiewicz function corresponding to the cubes
$\{ Q_k\}_{k=1}^\infty$ with center at $c_k$ and 
side-length $r_k$. 
Here each $\lambda_k$ is a positive constant.
We claim that if $w$ satisfies
the Dini condition and $1<q<\infty$, then 
$$
\int_{{\mathbb R}^n} F_w(x)^q dx \lesssim \sum_{k=1}^\infty \lambda_k{}^q|Q_k|
$$
where $\{ Q_k\}_{k=1}^\infty$ 
is a collection of disjoint cubes in ${\mathbb R}^n$ with $k \in {\mathbb N}$.

In fact,
we choose a non-negative function $g \in L^{q'}({\mathbb R}^n)$ with
$\|g\|_{L^{q'}} \le 1$ and we establish
$$
\int_{{\mathbb R}^n} F_w(x) g(x) dx \lesssim \sum_{k=1}^\infty \lambda_k{}^q|Q_k|
$$
We note that
\begin{eqnarray*}
\lefteqn{
\int_{{\mathbb R}^n} {\mathcal M}1_{Q(c_k,r_k)}(x) w\left(\frac{r_k}{r_k+|x-c_k|} \right) g(x) dx
}\\
&\lesssim& w(1)
\fint_{B(c_k,r_k)}g(x)dx
+
\sum_{l=1}^\infty
w(2^{-l})
\fint_{Q(c_k,2^l r_k)}g(x)dx\\
&\lesssim& 
\fint_{Q(c_k,r_k)}g(x)dx
+
\sum_{l=1}^\infty
w(2^{-l})
\fint_{Q(c_k,2^l r_k)}g(x)dx\\
&\lesssim& 
|B(c_k,r_k)|\inf_{Q(c_k,r_k)}{\mathcal M}g(y)
\end{eqnarray*}
Thus,
\[
\int_{{\mathbb R}^n} F_w(x) g(x) dx
\lesssim
\int_{{\mathbb R}^n} \sum_{k=1}^\infty \lambda_k 1_{Q(c_k,r_k)}(x) {\mathcal M}g(x) dx.
\]
If we use H\"{o}lder's inequality and then use the $L^{q'}$-boundedness
of ${\mathcal M}$,
then 
\begin{align*}
\int_{{\mathbb R}^n} F_w(x) g(x) dx
&\lesssim
\int_{{\mathbb R}^n} \sum_{k=1}^\infty \lambda_k 1_{B(c_k,r_k)}(x) {\mathcal M}g(x) dx\\
&\lesssim
\left\|\sum_{k=1}^\infty \lambda_k 1_{B(c_k,r_k)}\right\|_{L^q}
\|{\mathcal M}g\|_{L^{q'}}\\
&\lesssim
\left(\sum_{k=1}^\infty \lambda_k{}^q|Q_k|\right)^{\frac1q}
\|g\|_{L^{q'}},
\end{align*}
as required.
This leads to a generalization of some results in \cite{FStein}.
\end{remark}

\section*{Acknowledgements}

 The authors are grateful to the anonymous referees for their careful reading of the paper and for making several valuable comments which have improved the quality of this paper.

\end{document}